\documentclass{amsart}  
\usepackage{geometry}
\usepackage{amsfonts,amsmath}

\usepackage{algorithm}
\usepackage[noend]{algpseudocode}
\usepackage{graphicx}
\usepackage{tikz}
\usetikzlibrary{arrows}
\usepackage{tikzsymbols}
\usepackage{multirow}
\usepackage{xcolor}
\usepackage{comment}
\usepackage{enumitem}

\newtheorem*{remark}{Remark}
\newtheorem{theorem}{Theorem}[section]
\newtheorem{corollary}{Corollary}[theorem]
\newtheorem{lemma}[theorem]{Lemma}
\newtheorem{proposition}[theorem]{Proposition}
\newtheorem{assumption}[theorem]{Assumption}
\newtheorem{definition}[theorem]{Definition}

\newcommand{\bm}[1]{\mathbf{#1}}

\newcommand{\inner}[2]{\left\langle #1, #2 \right\rangle}
\newcommand{\abs}[1]{\left|#1\right|}

\DeclareMathOperator{\trace}{trace}

\newcommand{\qquant}[1]{q\text{-quant}\left(#1\right)}

\newcommand{\sigmin}[1]{\sigma_{\text{min}}(A)}
\newcommand{\sigminsquared}[1]{\sigma_{\text{min}}^2(A)}
\newcommand{\expect}[2]{\mathbb{E}_{#2}#1}
\newcommand{\conexpect}[3]{\mathbb{E}_{#2}\left[#1 \;|\; #3\right]}

\newcommand{\distfunc}{\Gamma (\mu,\sigma,\mu',\sigma')}

\newcommand{\noisediff}[1]{\frac{\lVert \vn_{B_{#1}\setminus S_{#1}}^{(#1)} \rVert_2^2}{\lVert A_{B_{#1}\setminus S_{#1}} \rVert_{F}^2}}
\newcommand{\noisediffb}[1]{\lVert \vn^{(#1)} \rVert_1^2}

\newcommand{\sysconst}{\zeta}
\newcommand{\qrkprob}{p}
\newcommand{\horizon}[1]{\gamma_{#1}}

\def\va{{\bm{a}}}
\def\vb{{\bm{b}}}
\def\vc{{\bm{c}}}

\def\vn{{\bm{n}}}

\def\vv{{\bm{v}}}

\def\vx{{\bm{x}}}

\title[On Quantile Randomized Kaczmarz for Time-Varying Noise and Corruption]{On Quantile Randomized Kaczmarz for Linear Systems with Time-Varying Noise and Corruption}

\author{Nestor Coria$^*$}
\address{}
\curraddr{}
\email{}
\thanks{The authors are grateful to and were partially supported by NSF DMS \#2211318. \\$^*$These authors contributed equally.}

\author{Jamie Haddock}
\address{Department of Mathematics, Harvey Mudd College,
Claremont, CA 91711}
\email{jhaddock@g.hmc.edu}

\author{Jaime Pacheco$^*$}
\email{}

\begin{document}
\maketitle

\begin{abstract}
    Large-scale systems of linear equations arise in machine learning, medical imaging, sensor networks, and in many areas of data science.  When the scale of the systems are extreme, it is common for a fraction of the data or measurements to be corrupted. The Quantile Randomized Kaczmarz (QRK) method is known to converge on large-scale systems of linear equations $A\vx=\vb$ that are inconsistent due to \emph{static} corruptions in the measurement vector $\vb$. We prove that QRK converges even for systems corrupted by \emph{time-varying} perturbations.
 Additionally, we prove that QRK converges up to a \emph{convergence horizon} on systems affected by time-varying noise and corruption.
 Finally, we utilize Markov's inequality to prove a lower bound on the probability that the largest entries of the QRK residual reveal the time-varying corruption in each iteration.  We present numerical experiments which illustrate our theoretical results.


\end{abstract}

\section{Introduction}
Solving large-scale systems of linear equations is one of the most commonly encountered problems across the data-rich sciences.  This problem arises in machine learning, as subroutines of several optimization methods~\cite{boyd2004convex}, in medical imaging~\cite{GBH70:Algebraic-Reconstruction,HM93:Algebraic-Reconstruction}, in sensor networks~\cite{savvides2001dynamic}, and in statistical analysis, to name only a few. Stochastic or randomized iterative methods have become increasingly popular approaches for a variety of large-scale data problems as these methods typically have low-memory footprint and are accompanied by attractive theoretical guarantees.  Indeed, the scale of modern problems often make application of direct or non-iterative methods challenging or infeasible~\cite{saad2003iterative}.  Examples of randomized iterative methods which have found popularity in recent years include the stochastic gradient descent method~\cite{robbins1951stochastic} and the randomized Kaczmarz method~\cite{SV09:Randomized-Kaczmarz}.

Large-scale data, where iterative methods typically find superiority over direct methods, nearly always suffers from corruption -- from error in measurement or data collection, in storage, in transmission and retrieval, and in data handling.  In this work, we consider solving linear problems where a few measurements or components have been subjected to corruptions that disrupt solution by usual techniques.
Problems in which a small number of untrustworthy data can have a devastating effect on a variable of interest have been considered in ~\cite{awasthi2014power,lai2016agnostic,charikar2017learning,diakonikolas2019robust}.
In particular, linear problems with a small number of outlier measurements have led to an interest in methods for robust linear regression~\cite{rousseeuw1984least,vivsek2006least,NIPS2017_e702e51d}.
Other relevant work includes min-k loss stochastic gradient descent (SGD)~\cite{pmlr-v108-shah20a}, robust SGD~\cite{diakonikolas2019sever,prasad2020robust}, and Byzantine approaches~\cite{blanchard2017machine,alistarh2018byzantine}.

In application areas where robust linear regression and robust numerical linear algebra are most impactful, data perturbations are often introduced in a streaming and non-static manner (e.g., during distributed data access~\cite{li2022distributed}). For example, when solving a system of linear equations that is large enough to be distributed amongst various storage servers, accessing subgroups of equations at different times from these servers may introduce a varying amount of noise and corruption (or none) into the accessed measurement; \emph{a single equation could be corrupted or noisy in one access but not in another}.   The previously considered model of static corruption in the measurement vector is no longer sufficient; the measurement vector, noise, and corruption should be modeled in a time-varying manner. Problems with time-varying perturbation also arise in sensor networks~\cite{savvides2001dynamic} and medical imaging~\cite{HM93:Algebraic-Reconstruction,GBH70:Algebraic-Reconstruction}.

\subsection{Problem set-up}\label{subsec:problem}

In this paper, we consider solving highly overdetermined linear systems which are perturbed by both noise and corruption that vary with time. We are considering the situation in which we have a matrix $A \in \mathbb{R}^{m \times n}$ and a vector $\vb \in \mathbb{R}^m$ which define a highly overdetermined consistent system $A\vx = \vb$, with $m \gg n$.  We assume the measurement matrix is full-rank, $\text{rank}(A) = n$, and let $\vx^* \in \mathbb{R}^n$ denote the solution to the consistent system $A\vx^* = \vb$.  We will represent the noise and corruption as vectors whose value varies with the iteration; the noise has vector value $\vn^{(k)} \in \mathbb{R}^m$ in the $k$th iteration and the corruption has vector value $\vc^{(k)} \in \mathbb{R}^m$ in the $k$th iteration. These vectors obscure the true value of $\vb$, creating the observed, time-varying measurement vector $\vb^{(k)} = \vb + \vn^{(k)} + \vc^{(k)}$, which forms an inconsistent system with $A$.
 We assume that the values of $\vb, \vn^{(k)},$ and $\vc^{(k)}$ are not known and only $\vb^{(k)}$ is observed.
We assume that the number of corruptions is no more than a fraction $0 < \beta \ll 1$ of the total number of measurements, $\|\vc^{(k)}\|_{\ell_0} := |\text{supp}(\vc^{(k)})| \le \beta m$. Here $\text{supp}(\vx)$ denotes the set of indices of nonzero entries of $\vx$.  We define the set of corrupted equations in the $k$th iteration to be $C_{k} = \text{supp}(\vc^{(k)}) \subset \{ 1,2,\dots,m \}.$

\subsection{Organization}
The remainder of our paper is organized as follows. We discuss related works in Section~\ref{subsec:related work}, define relevant notation and state our assumptions in Section~\ref{subsec:notation}, provide detailed pseudocode of the \emph{quantile randomized Kaczmarz (QRK)} method in Section~\ref{subsec:QRK}, and present our main contributions in Section~\ref{subsec:contributions}.  In Section~\ref{sec:theoretical results} we present statements and proofs of our theoretical results, and in Section~\ref{sec:numerical experiments} we empirically demonstrate our theoretical results with numerical experiments on synthetic data.  Finally, we provide some conclusions and discussion in Section~\ref{sec:conclusion}.

\subsection{Related work}\label{subsec:related work}

The Kaczmarz methods, which are a classical example of \emph{row-action} iterative methods, produce iterates as sequential orthogonal projections towards the solution set of a single equation~\cite{Kac37:Angenaeherte-Aufloesung}; the $j$th iterate is recursively defined as
\begin{equation}
    \vx^{(j)} = \vx^{(j-1)} - \frac{\va_{i_{j}}^T \vx^{(j-1)} - b_{i_j}}{\|\va_{i_{j}}\|^2} \va_{i_j} \label{eq:RKupdate}
\end{equation}
where $\va_{i_j}$ is the $i_j$th row of $A$, $b_{i_j}$ is the $i_j$th entry of $\vb$, and $i_j \in \{1, 2, \cdots, m\}$ is the equation selected in the $j$th iteration.
These methods saw a renewed surge of interest after the elegant convergence analysis of the \emph{Randomized Kaczmarz (RK) method} in~\cite{SV09:Randomized-Kaczmarz}; the authors showed that for a consistent system with unique solution $\vx^*$, RK (where $i_j = l$ with probability $\|\va_l\|^2/\|A\|_F^2$) converges at least linearly in expectation with the guarantee
\begin{equation}
    \mathbb{E}\|\vx^{(k)} - \vx^*\|^2 \le \left(1 - \frac{\sigma_{\min}^2(A)}{\|A\|_F^2}\right)^k \|\vx^{(0)} - \vx^*\|^2, \label{eq:RKrate}
\end{equation}
where $\sigma_{\min}(A)$ is the minimum singular-value of the matrix $A$.  Many variants and extensions followed, including convergence analyses for inconsistent and random linear systems~\cite{Nee10:Randomized-Kaczmarz, CP12:Almost-Sure-Convergence}, connections to other popular iterative algorithms~\cite{Ma2015convergence, NSWjournal, Pop01:Fast-Kaczmarz-Kovarik, Pop04:Kaczmarz-Kovarik-Algorithm, frek}, block approaches~\cite{needell2013paved, popa2012kaczmarz}, acceleration and parallelization strategies~\cite{EN11:Acceleration-Randomized, liu2014asynchronous, morshed2019accelerated, moorman2020randomized}, and techniques for reducing noise~\cite{ZF12:Randomized-Extended}.

A recent line of work has introduced \emph{quantile statistics} of the in-iteration residual information into randomized iterative methods to detect and avoid corrupted data in large-scale linear systems~\cite{HNRS20}.  These problems arise in medical imaging, sensor networks, error correction and data
science, where effects of adversarial corruptions in the data could be catastrophic down-stream.  An ill-timed update using corrupted data can destroy the valuable information learned by many updates produced with uncorrupted data, making iterative methods challenging to employ on large-scale data; see Figure~\ref{fig:corrlinsys} for an example of the effect of corrupted data on an iterative method for solving linear systems.  The \emph{quantile randomized Kaczmarz (QRK)} method is a variant of the Kaczmarz method designed to avoid the impacts of a small number of corruptions in the measurement vector.  QRK differs from RK in that, in each iteration, the usual Kaczmarz projection step is only taken if the sampled residual entry magnitude is less than a sufficient fraction of residual entry magnitudes, indicating that this sampled equation is likely not affected by a corruption.  QRK is known to converge in expectation to the solution to the uncorrupted system $\vx$ under mild assumptions~\cite{HNRS20, steinerberger2021quantile, jn21qrk,blockqrk22,HMR23}.  By ensuring that updates are generated using only data with residual magnitude less than a sufficiently-small quantile statistic of the entire residual, QRK can avoid updates corresponding to highly corrupted data.  In this paper, we show that QRK converges in the more general case that the system is affected by time-varying corruption and noise.

\begin{figure}
\begin{minipage}{0.4\textwidth}
\definecolor{uuuuuu}{rgb}{0.26666666666666666,0.26666666666666666,0.26666666666666666}
\definecolor{qqqqff}{rgb}{0.,0.,1.}
\begin{tikzpicture}[scale=0.7,line cap=round,line join=round,>=triangle 45,x=0.7cm,y=0.7cm]
\clip(-4.761865744202146,-0.995090526547853) rectangle (8.23893530695861,7.433382978997327);
\draw [domain=-6.761865744202146:8.23893530695861] plot(\x,{(--7.3696--1.56*\x)/1.1});
\draw [domain=-6.761865744202146:8.23893530695861] plot(\x,{(--9.996--1.62*\x)/1.74});
\draw [domain=-6.761865744202146:8.23893530695861] plot(\x,{(--8.8592--1.64*\x)/1.44});
\draw [domain=-6.761865744202146:8.23893530695861] plot(\x,{(--11.1328--1.6*\x)/2.04});
\draw [domain=-6.761865744202146:8.23893530695861] plot(\x,{(--12.2696--1.58*\x)/2.34});
\draw [domain=-6.761865744202146:8.23893530695861,color=blue] plot(\x,{(--16.4664-4.4*\x)/-3.38});
\draw [domain=-6.761865744202146:8.23893530695861,color=blue] plot(\x,{(--22.1136-4.94*\x)/-2.82});
\draw (-1.9838328167953874,4.022105723941761) node[anchor=north west] {$\vx$};
\draw (5.627684753608404,0.266333587887185) node[anchor=north west] {$\vx_0$};
\draw [dash pattern=on 5pt off 5pt] (5.618740594239209,-0.07862180595980683)-- (-0.6432106064704927,5.419676809297492);
\draw (-1.7082475926783225,6.349991791969049) node[anchor=north west] {$\vx_1$};
\draw [dash pattern=on 5pt off 5pt] (-0.6432106064704927,5.419676809297492)-- (4.733088780528801,1.2897013711025815);
\draw (4.066527656226953,2.3886477826500872) node[anchor=north west] {$\vx_2$};
\begin{scriptsize}
\draw [fill=qqqqff] (-1.96,3.92) circle (2.5pt);
\draw [fill=black] (5.618740594239209,-0.07862180595980683) circle (1.5pt);
\draw [fill=uuuuuu] (-0.6432106064704927,5.419676809297492) circle (1.5pt);
\draw [fill=uuuuuu] (4.733088780528801,1.2897013711025815) circle (1.5pt);
\end{scriptsize}
\end{tikzpicture}
\end{minipage}
\caption{A linear system in which some equations (blue) have been corrupted and two iterations of a Kaczmarz method.}\label{fig:corrlinsys}
\end{figure}
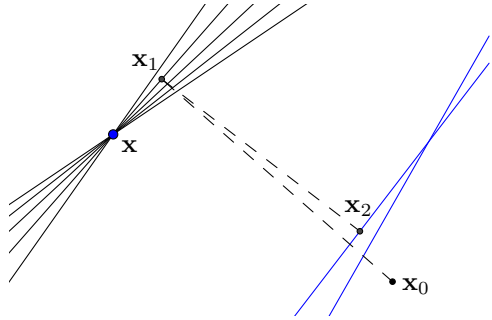

The setting of time-varying and random noise in the system measurements has been considered previously in the Kaczmarz literature.  In~\cite{marshall2023optimal}, the authors consider the scenario that the measurements, $\vb$, of a consistent system are perturbed by independent mean zero random noise with bounded variance, and derive an optimal learning rate (or relaxation parameter) schedule for the randomized Kaczmarz method which minimizes the expected error.  In~\cite{tondji2023adaptive}, the authors consider an \emph{independent noise} model in which, each time the system is sampled, the entries of the sample are assumed to be perturbed by freshly sampled independent mean zero random noise values with bounded variance, and derive adaptive step sizes that allow the block Bregman-Kaczmarz method to converge to the solution of the unperturbed system.  Both of these works develop methods which converge to the solution of the unperturbed (no noise) system, breaking the so-called \emph{convergence horizon} suffered by the standard RK method applied to systems with noise~\cite{Nee10:Randomized-Kaczmarz}.  On the other hand, our paper will be concerned with quantifying the effect of general time-varying noise and corruption on the convergence horizon and rate of QRK.

\subsection{Notation and assumptions}\label{subsec:notation}
We will use boldfaced lower-case Latin letters (e.g., $\vx$) to denote vectors and unbolded upper-case Latin letters (e.g., $A$) to denote matrices.  We use unbolded lower-case Latin and Roman letters (e.g., $q$ and $\beta$) to denote scalars.  We let $[m]$ denote the set $\{1, 2, \cdots, m\}$.  Given a vector $\vv \in \mathbb{R}^n$, we let $v_i$ denote the $i$th component of $\vv$. Throughout this paper, we will consider only real-valued matrices and vectors. If $A$ is an $m\times n$ matrix and $S\subset [m],$ then let $A_S$ denote the matrix obtained by restricting to the rows indexed by $S$ and let $A_{\not\in S}$ denote the matrix obtained by removing from $A$ the rows indexed by $S$.  We denote the $i$th row of $A$ by $\va_i \in \mathbb{R}^n$ for $i \in [m]$. For $q \in [0,1]$, we let $\qquant{S}$ denote the empirical $q$-quantile of a multi-set $S$, the $\lfloor  q|S| \rfloor$-\emph{th} smallest element of $S$; that is, $s \in S$ such that $|\{r \in S: r \le s\}| = \lfloor  q|S| \rfloor$.

The notation $\|\vv\|$ denotes the Euclidean norm of a vector $\vv$. All other $p$-norms of vectors $\vv$ are denoted with subscript, $\|\vv\|_p$.

For a matrix $A$, we denote its operator norm by $\|A\| = \sup_{\vx \in S^{n-1}} \|A\vx\|$ and its Frobenious (or Hilbert-Schmidt) norm by $\|A\|_F =\sqrt{\trace(A^\top A)}$.  Throughout, we denote by $\sigmin{A}$ and $\sigma_{\max}(A)$ the smallest and largest singular values of the matrix $A$ (that is, eigenvalues of the matrix $\sqrt{A^\top A}$).  Moreover, we always assume that the matrix $A$ has full column rank, so that $\sigma_{\min}(A) > 0$.

Our convergence analyses involve several random variables; the random sample of index in the $k$th iteration, $i_k$, and in some cases, the random noise introduced into the system measurements in the $k$th iteration, $\vn^{(k)}$. In particular, the algorithms will make use of the uniform distribution, $\text{Unif}(\cdot)$.  For a finite multi-set $S$, we have $\text{Unif}(S)$ denote the uniform distibution over the elements of $S$; that is, $s \in S$ is sampled with probability $1/|S|$.

We will be interested in $\expect{\|\vx^{(k+1)} - \vx^*\|^2}{i_k}$ which is the expectation of the squared error of $\vx^{(k+1)}$ with respect to the sampled index in the $k$th iteration, $i_k$, where the sampling is according to the distribution in context, i.e., the Kaczmarz sampling distribution.

We will also be interested in $\conexpect{\|\vx^{(k+1)} - \vx^*\|^2}{i_k}{i_k \in S_{k}}$ which is the conditional expectation of the squared error of $\vx^{(k+1)}$ with respect to the sampled index in the $k$th iteration, $i_k$, conditioned upon $i_k \in S_{k}$, where the sampling is according to the distribution in context, i.e., the Kaczmarz sampling distribution.

We'll let $\expect{\|\vx^{(k+1)} - \vx^*\|^2}{i_1, \cdots, i_k}$ denote the iterated expectation with respect to all sampled indices and will let $\expect{\|\vx^{(k+1)} - \vx^*\|^2}{}$ denote the expectation with respect to the joint distribution of all random variables.

Our convergence results will hinge upon a submatrix conditioning parameter introduced in~\cite{steinerberger2021quantile}; we define
\begin{equation}
    \sigma_{q-\beta,\min}(A) = \min_{\substack{S \subset [m]\\|S| = \lfloor(q-\beta)m\rfloor}} \inf_{\|\vx\| = 1} \|A_S\vx\|.
\end{equation}
This parameter measures the smallest singular value encountered in any submatrix determined by a $q-\beta$-fraction of the rows of $A$.  Our results will also depend upon three parameters which determine the convergence rate and convergence horizon; we define these terms in Definition~\ref{def:conv_params}.

\begin{definition}\label{def:conv_params}
We define a rate parameter,
    \[\varphi = \frac{\sigma_{q-\beta,\text{\normalfont{min}}}^2(A)}{q m}\bigg(\frac{q-\beta}{q}\bigg) - \frac{\sigma_{\text{\normalfont{max}}}^2(A)}{q m}\bigg(\frac{2\sqrt{\beta(1-\beta)}}{(1-q-\beta)}+ \frac{\beta(1-\beta)}{(1-q-\beta)^2}\bigg) - \frac{\sigma_{\text{\normalfont{max}}}(A)}{q m}\bigg( \frac{\sqrt{\beta m}}{m(1-q-\beta)}+\frac{\beta\sqrt{m(1-\beta)}}{m(1-q-\beta)^2} \bigg), \]
a parameter governing the influence of the quantile and corruption rate on the convergence horizon,
    \[\sysconst = \frac{\sigma_{\text{\normalfont{max}}}(A)}{q m}\bigg( \frac{\sqrt{\beta m}}{m(1-q-\beta)}+\frac{\beta\sqrt{m(1-\beta)}}{m(1-q-\beta)^2} \bigg) +\frac{\beta}{q m^2(1-q-\beta)^2},\]
and a parameter measuring the convergence horizon,
\[\horizon{k} = \sum_{j=0}^{k} (1 - p\varphi)^{k-j} \left(\frac{\|\vn^{(k)}\|_2^2}{(q-\beta)m} + \sysconst \noisediffb{j}\right).\]
Here, $p$ is a fixed probability defined below.
\end{definition}

We list next two assumptions made throughout this paper.

\begin{assumption}\label{ass:unit_rows}
    The matrix $A \in \mathbb{R}^{m \times n}$ has rows with unit norm, $\|\va_i\| = 1$ for all $i \in [m]$.
\end{assumption}

\begin{assumption}\label{ass:pos_rate}
    The rate parameter is positive, $\varphi > 0$.
\end{assumption}
We note that Assumption~\ref{ass:pos_rate} is similar to the assumption $c_{A,\beta,q} > 0$ in \cite[Theorem (Main Result)]{steinerberger2021quantile}.

\subsection{Quantile randomized Kaczmarz}\label{subsec:QRK}
As mentioned previously, QRK is a variant of RK where, in each iteration, the update is made only if the sampled residual entry magnitude is less than a sufficient fraction of residual entry magnitudes.  The algorithm uses this residual magnitude comparison as a proxy indicator for whether it is affected by a corruption.

Practically, there are two implementations of the QRK algorithm.  In the first, we randomly sample from the entire residual and only make the update if the sampled residual magnitude is less than a sufficient fraction of the other residual entry magnitudes. We note that a fraction of Algorithm~\ref{QuantileRK1} iterations will not update the iterates; this behavior is key for this corruption-robust algorithm. Pseudocode for this implementation is given in Algorithm~\ref{QuantileRK1}; this implementation is due to~\cite{HNRS20}.  In the second, we randomly sample only from the residual entries whose magnitude is less than a sufficient fraction of the other residual entry magnitudes and make the update in each iteration.  Pseudocode for this implementation is given in Algorithm~\ref{QuantileRK2}; this implementation is due to~\cite{steinerberger2021quantile}. We note that one of our goals is a unified analysis of Algorithms~\ref{QuantileRK1} and~\ref{QuantileRK2}.
\begin{algorithm}
	\caption{QuantileRK~\cite{HNRS20}}\label{QuantileRK1}
	\begin{algorithmic}[1]
		\Procedure{QuantileRK1}{$A,\vb$, $\vx^{(0)}$, q, N}
		\For{j = 1, \ldots, N}
		\State{sample $i_j\sim\text{Unif}([m])$}
		\If{ $\abs{\inner{\va_{i_j}}{\vx^{(j-1)}} - b_{i_j}} \leq \qquant{\left\{\abs{\inner{\va_i}{\vx^{(j-1)}} - b_i}\right\}_{i=1}^m} $}
		\State{$\vx^{(j)} = \vx^{(j-1)} - \frac{\left(\inner{\va_{i_j}}{\vx^{(j-1)}} - b_{i_j}\right)}{\|\va_{i_j}\|^2} \va_{i_j}$}
		\Else
		\State{$\vx^{(j)} = \vx^{(j-1)}$}

		\EndIf

		\EndFor{}

		\Return{$\vx^{(N)}$}
		\EndProcedure
	\end{algorithmic}
\end{algorithm}

\begin{algorithm}
	\caption{QuantileRK~\cite{steinerberger2021quantile}}\label{QuantileRK2}
	\begin{algorithmic}[1]
		\Procedure{QuantileRK2}{$A,\vb$, $\vx^{(0)}$, q, N}
		\For{j = 1, \ldots, N}
		\State{sample $i_j\sim\text{Unif}\left(\left\{l : \abs{\inner{\va_l}{\vx^{(j-1)}} - b_l} \leq \qquant{\left\{\abs{\inner{\va_i}{\vx^{(j-1)}} - b_i}\right\}_{i=1}^m} \right\}\right)$}
		\State{$\vx^{(j)} = \vx^{(j-1)} - \frac{\left(\inner{\va_{i_j}}{\vx^{(j-1)}} - b_{i_j}\right)}{\|\va_{i_j}\|^2} \va_{i_j}$}

		\EndFor{}

		\Return{$\vx^{(N)}$}
		\EndProcedure
	\end{algorithmic}
\end{algorithm}

These implementations may be analyzed in the same way; we will provide a convergence analysis valid for both in Theorem~\ref{thm:steinerbergerQRKwNoise} and provide a proof in Section~\ref{sec:theoretical results}.

\subsection{Main contributions}\label{subsec:contributions}

In this section, we list our main results.  Our first main result proves that QRK converges at least linearly in expectation up to a convergence horizon even in the case of time-varying corruption and noise.  We note that our bound on the expected error includes two terms; the first term decreases linearly and depends upon $\varphi$, which is independent of the noise $\vn^{(k)}$, and the second term defines the convergence horizon for the method and depends upon $\horizon{k}$, which does depend on the noise $\vn^{(k)}$.

\begin{theorem} \label{thm:steinerbergerQRKwNoise}
Let $\vx^{(k)}$ denote the iterates of Algorithm~\ref{QuantileRK1} or Algorithm~\ref{QuantileRK2} applied with quantile $q$ to the system defined by $A$ and $ \vb^{(k)} = \vb + \vn^{(k)} + \vc^{(k)}$, in the $k$th iteration. Assuming the setup described in Section~\ref{subsec:problem}, Assumptions~\ref{ass:unit_rows} and~\ref{ass:pos_rate}, and $\beta<q<1-\beta$ arbitrary, then
\[
\expect{\lVert \vx^{(k+1)}-\vx^*\rVert^2}{\;} \leq (1 - p\varphi)^{k+1}\lVert\vx^{(0)}-\vx^*\rVert^2 + p\horizon{k},
\]
\noindent where $p$ represents the probability that a row is selected from within the $q$-quantile in each iteration, \[p = \begin{cases} q & \text{ applying QuantileRK1 (Algorithm~\ref{QuantileRK1})}\\ 1 & \text{ applying QuantileRK2 (Algorithm~\ref{QuantileRK2})} \end{cases},\] and $\varphi$ and $\horizon{k}$ are as given in Definition~\ref{def:conv_params}.
\end{theorem}

\begin{remark}
    Now, we note that we have $0 < 1 - p\varphi < 1$. First, by definition of $\sigma_{q-\beta,\min}(A)$ and using the Cauchy-Schwarz inequality in conjunction with Assumption~\ref{ass:unit_rows}, we have $\sigma_{q-\beta,\min}(A) \le \sqrt{(q - \beta)m}$ and thus, $\varphi \le \left(\frac{q - \beta}{q}\right)^2 < 1$, so $ 1 - p\varphi > 0$.  On the other hand, by Assumption~\ref{ass:pos_rate}, we have $ 1 - p\varphi < 1$.
\end{remark}

In Corollary~\ref{cor:single_noisy_cor}, we take stronger assumptions on the form of the noise.  In Corollary~\ref{cor:single_noisy_cor}~\ref{cor:boundednoisecorr}, we specialize Theorem~\ref{thm:steinerbergerQRKwNoise} to the case that the noise in each iteration is uniformly bounded.  In Corollary~\ref{cor:single_noisy_cor}~\ref{cor:defnoise}, we specialize Theorem~\ref{thm:steinerbergerQRKwNoise} to the case that the noise is sampled from a known distribution, and in Corollary~\ref{cor:single_noisy_cor}~\ref{cor:gaussian noise}, we specialize to the case that the noise has entries sampled from a Gaussian distribution.

\begin{corollary}\label{cor:single_noisy_cor}
Let $\vx^{(k)}$ denote the iterates of Algorithm~\ref{QuantileRK1} or Algorithm~\ref{QuantileRK2} applied with quantile $q$ to the system defined by $A$ and $ \vb^{(k)} = \vb + \vn^{(k)} + \vc^{(k)}$, in the $k$th iteration.  We assume the setup described in Section~\ref{subsec:problem}, Assumptions~\ref{ass:unit_rows} and~\ref{ass:pos_rate}, and $\beta<q<1-\beta$ arbitrary.  We let $p$ be the probability that a row is selected from within the $q$-quantile in each iteration, \[p = \begin{cases} q & \text{ applying QuantileRK1 (Algorithm~\ref{QuantileRK1})}\\ 1 & \text{ applying QuantileRK2 (Algorithm~\ref{QuantileRK2})} \end{cases}\] and $\varphi$ and $\sysconst$ are as given in Definition~\ref{def:conv_params}.
\begin{enumerate}[label=(\alph*)]
    \item If $n_{max} \geq \|\vn^{(j)}\|_\infty$ for all $j \in [k+1]$, then \[
\expect{\| \vx^{(k+1)}-\vx^*\|^2}{i_1,\cdots,i_k} \leq (1 - \qrkprob \varphi)^{k+1}\lVert\vx^{(0)}-\vx^*\rVert^2 + (1 + \sysconst m^2)(n_{max})^2 \frac{1-(1 - \qrkprob \varphi)^{k+1}}{\varphi}.
\] \label{cor:boundednoisecorr}
\item If $n_i^{(k)}$ are i.i.d samples from distribution with mean $\mu$ and standard deviation $s$, and the mean and standard deviation of $|n_i^{(k)}|$ are $\mu'$ and $s'$, respectively, then \[
\expect{\| \vx^{(k+1)}-\vx^*\|^2}{} \leq (1-\qrkprob \varphi)^{k+1}\lVert\vx^{(0)}-\vx^*\rVert^2 + \frac{1-(1-\qrkprob \varphi)^{k+1}}{\varphi} \left({\mu^2 + s^2} + \sysconst [m^2(\mu')^2 + m(s')^2]\right).
\] \label{cor:defnoise}
\item If $n_i^{(k)}$ are i.i.d samples from $\mathcal{N}(0, s^2)$, then \[
    \expect{\| \vx^{(k+1)}-\vx^*\|^2}{} \leq (1-\qrkprob \varphi)^{k+1}\lVert\vx^{(0)}-\vx^*\rVert^2 + \frac{1-(1-\qrkprob \varphi)^{k+1}}{\varphi} s^2(1 + \sysconst [m^2 \frac{2}{\pi} + m (1-\frac{2}{\pi})]).
\] \label{cor:gaussian noise}
\end{enumerate}
\end{corollary}

Our final main result utilizes Theorem~\ref{thm:steinerbergerQRKwNoise} to provide a lower bound on the probability that the indices of the corrupted equations in the $k$th iteration are given by examining the $k$th residual and identifying the $\lfloor \beta m \rfloor$ largest entries.  That is, \emph{even in the case that the corruption indices are changing and the system is perturbed by time-varying noise, the largest entries of the residual reveal the position of the corruptions}.

\begin{corollary}\label{cor:noisy corruption detection}
Let $\vx^{(k)}$ denote the iterates of Algorithm~\ref{QuantileRK1} or Algorithm~\ref{QuantileRK2} applied with quantile $q$ to the system defined by $A$ and $ \vb^{(k)} = \vb + \vn^{(k)} + \vc^{(k)}$, in the $k$th iteration.
Let $n_{max}$ be the value that satisfies $n_{max} \geq \|\vn^{(j)}\|_\infty$ for all $1 \leq j \leq k+1$, and $\varphi$ and $\sysconst$ are as given in Definition~\ref{def:conv_params}.
Define $c^{(k)}_{\min} = \min_{i \in \text{supp}(\vc^{(k)})} |c_i^{(k)}|,$ and
assume $(1 + \sysconst m^2)\frac{n_{\max}^2}{\varphi}  \le (c_{\min}^{(k)})^2/4M$ for some value $M \gg 1$.
Assuming the setup described in Section~\ref{subsec:problem}, Assumptions~\ref{ass:unit_rows} and~\ref{ass:pos_rate}, and $\beta<q<1-\beta$ arbitrary, then with probability at least
\[\frac{M-1}{M} - \frac{4(1 - p \varphi)^k \|\vx^{(0)} - \vx^*\|^2}{(c^{(k)}_{\min})^2} \]
 where $p$ represents the probability that a row is selected from within the $q$-quantile in each iteration, \[p = \begin{cases} q & \text{ applying QuantileRK1 (Algorithm~\ref{QuantileRK1})}\\ 1 & \text{ applying QuantileRK2 (Algorithm~\ref{QuantileRK2})}, \end{cases}\] we have that by iteration $k$, the (time-varying) position of the corruptions can be identified by evaluating the residual and identifying the $\beta m$ largest entries; that is,
\[\text{supp}(\vc^{(k)}) \subset \text{argmax}_{\substack{S \subset [m], \\|S| = \lfloor\beta m\rfloor}} \sum_{i \in S} |\va_i^\top \vx^{(k)} - b_i^{(k)}|.\]
\end{corollary}

\begin{remark}
    In the noiseless case, when $\vn^{(k)} = \mathbf{0}$ for all iterations $k$, Corollary~\ref{cor:noisy corruption detection} holds with probability at least \[1 - \frac{4(1 - p \varphi)^k \|\vx^{(0)} - \vx^*\|^2}{(c^{(k)}_{\min})^2}. \]
\end{remark}

\section{Theoretical Results}\label{sec:theoretical results}

In this section, we provide proofs of our main results.  In Section~\ref{subsec:noisy RK}, we provide a convergence analysis for the randomized Kaczmarz method~\cite{SV09:Randomized-Kaczmarz} on a system perturbed by time-varying noise with entries sampled i.i.d.~ from a distriution with mean $\mu$ and standard deviation $s$.  This result will be utilized as a lemma in several of our other proofs, but may be of independent interest.  In Section~\ref{subsec:proof of main}, we prove our main result, Theorem~\ref{thm:steinerbergerQRKwNoise}, and in Section~\ref{subsec:proofs of corollaries}, we prove the corollaries of our main result, Corollary~\ref{cor:single_noisy_cor} and Corollary~\ref{cor:noisy corruption detection}.

\subsection{Convergence of RK on systems with time-varying noise}\label{subsec:noisy RK}

Our convergence analysis of QRK will hinge upon the approach taken in analyzing the randomized Kaczmarz method in the case that the system of interest is perturbed by time-varying noise.  This result, while useful as a lemma to us, is of independent interest.

\begin{theorem}\label{thm:noisyRK}
Let $A \in \mathbb{R}^{m \times n}$ with $m>n$ and $\vb \in \mathbb{R}^m$ define a consistent system $A\vx^* = \vb$ with unique solution $\vx^* \in \mathbb{R}^n$. Let $\vx^{(k)}$ denote the iterates of the randomized Kacmzarz method~\cite{SV09:Randomized-Kaczmarz} applied to the system defined by $A$ and $\vb^{(k)} = \vb + \vn^{(k)}$ in the $k$ iteration.

If $n_i^{(k)}$ are i.i.d samples from a distribution with mean $\mu$ and standard deviation $s$, then the error decreases at least linearly in expectation with
\[
\expect{\| \vx^{(k+1)}-\vx^*\|^2}{} \leq \phi^{k+1} \|\vx^{(0)}-\vx^*\|^2 + \frac{1-\phi^{k+1}}{1-\phi} \frac{m}{\|A\|^2_F}(s^2+\mu^2)
\]
where
\[
\phi = 1 - \frac{\sigma_{min}^2(A)}{\|A\|^2_F}.
\]
\end{theorem}

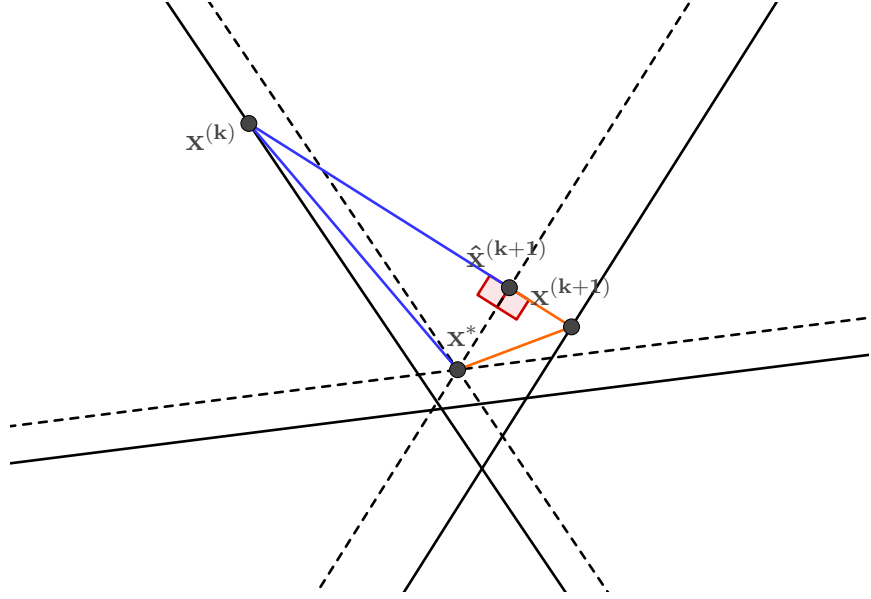
\begin{figure}
    \begin{center}
    \definecolor{ttttff}{rgb}{0.2,0.2,1}
\definecolor{ccqqqq}{rgb}{0.8,0,0}
\definecolor{ffwwqq}{rgb}{1,0.4,0}
\definecolor{uuuuuu}{rgb}{0.26666666666666666,0.26666666666666666,0.26666666666666666}
\begin{tikzpicture}[line cap=round,line join=round,>=triangle 45,x=1cm,y=1cm,scale=1.5]
\clip(-3.871828266237992,-1.335625506296871) rectangle (3.735920479096977,3.889656027791738);

\draw[line width=1pt,color=ccqqqq,fill=ccqqqq,fill opacity=0.1] (0.38059938088840994,1.4694288381386829) -- (0.27151820922573516,1.296539757858026) -- (0.4444072895063919,1.1874585861953513) -- (0.5534884611690667,1.3603476664760081) -- cycle;
\draw[line width=1pt,color=ccqqqq,fill=ccqqqq,fill opacity=0.10000000149011612] (0.4444072895063919,1.1874585861953513) -- (0.6172963697870487,1.0783774145326765) -- (0.7263775414497234,1.2512664948133334) -- (0.5534884611690667,1.3603476664760081) -- cycle;

\draw [dashed, line width=1pt, domain=-4.871828266237992:4.735920479096977] plot(\x,{(--1.6095084563946158-3.086293567338196*\x)/2.08324192165157});
\draw [dashed, line width=1pt, domain=-4.871828266237992:4.735920479096977] plot(\x,{(-0.9265458184069523-3.0398707143766526*\x)/-1.917950275916625});
\draw [dashed, line width=1pt, domain=-4.871828266237992:4.735920479096977] plot(\x,{(--1.8037767505481548--0.3672707237751681*\x)/2.905588921873659});
\draw [line width=1pt,domain=-4.871828266237992:4.735920479096977]  plot(\x,{(--0.4506289062066884-3.086293567338196*\x)/2.08324192165157});
\draw [line width=1pt,domain=-4.871828266237992:4.735920479096977] plot(\x,{(--1.4134014617303752-3.0398707143766526*\x)/-1.917950275916625});

\draw [line width=1pt,color=ffwwqq] (1.1040686947047949,1.0129692382323496)-- (0.09441107812634106,0.6327292762248319);
\draw [line width=1pt,color=ttttff] (-1.755619208012084,2.8172364973996387)-- (0.09441107812634106,0.6327292762248319);
\draw [line width=1pt,color=ttttff] (-1.755619208012084,2.8172364973996387)-- (0.5534884611690667,1.3603476664760081);
\draw [line width=1pt,color=ffwwqq] (0.5534884611690667,1.3603476664760081)-- (1.1040686947047949,1.0129692382323496);
\draw [line width=1pt,domain=-4.871828266237992:4.735920479096977] plot(\x,{(.8482283028167619--0.3672707237751681*\x)/2.905588921873659});
\begin{scriptsize}
\draw [fill=uuuuuu] (0.09441107812634106,0.6327292762248319) circle (2pt);
\draw[color=uuuuuu] (0.13741756254922948,0.9360201296521732) node {\Large $\bm{x}^*$};
\draw[color=black] (-2.3084770463191338,3.7691977937729955);
\draw [fill=uuuuuu] (1.1040686947047949,1.0129692382323496) circle (2pt);
\draw[color=uuuuuu] (1.102900141091679,1.3311231372336485) node {\Large $\bm{x^{(k+1)}}$};
\draw [fill=uuuuuu] (0.5534884611690667,1.3603476664760081) circle (2pt);
\draw[color=uuuuuu] (0.5367071096746993,1.6684061924861273) node {\Large $\bm{\hat{x}^{(k+1)}}$};
\draw [fill=uuuuuu] (-1.755619208012084,2.8172364973996387) circle (2pt);
\draw[color=uuuuuu] (-2.08555709376905,2.6995286756865626) node {\Large $\bm{x^{(k)}}$};
\end{scriptsize}
\end{tikzpicture}
    \end{center}
    \caption{An illustration of the errors of the iterates $\vx^{(k+1)}$ and $\hat{\vx}^{(k+1)}$ produced from $\vx^{(k)}$ on the noisy and noiseless systems, respectively.}
    \label{fig:kacz_diagram}
\end{figure}

\begin{proof}
We apply here a proof technique inspired by~\cite{Nee10:Randomized-Kaczmarz}.  Consider two instances of the randomized Kaczmarz method producing iterates $\{\vx^{(k)}\}$ and $\{\hat{\vx}^{(k)}\}$. One sequence of iterates, $\{\vx^{(k)}\}$, is generated using the time varying inconsistent system defined by $A$ and $\vb^{(k)} = \vb + \vn^{(k)}$, and the other, $\{\hat{\vx}^{(k)}\}$ using the unobserved noiseless system defined by $A$ and $\vb$, both with the same sequence of row selections $\{i_k\}$. As illustrated in the diagram in Figure~\ref{fig:kacz_diagram}, the Pythagorean theorem and the fact that the randomized Kaczmarz method updates via orthogonal projection implies
\begin{equation}
\|\vx^{(k)}-\vx^*\|^2 = \|\vx^{(k)}-\hat{\vx}^{(k+1)}\|^2+\|\hat{\vx}_{(k+1)}-\vx^*\|^2, \label{triangle1}
\end{equation}
and
\begin{equation}
\|\vx^{(k+1)}-\vx^*\|^2 = \|\hat{\vx}^{(k+1)}-\vx^*\|^2 + \|\hat{\vx}^{(k+1)}-\vx^{(k+1)}\|^2. \label{triangle2}
\end{equation}

We now observe that the quantity $\|\hat{\vx}^{(k+1)}-\vx^{(k+1)}\|$ represents the difference between a non-noisy and noisy iteration of Kaczmarz, where both use $\vx^{(k)}$ as the previous iterate.  Then
\begin{eqnarray*}
{\|\hat{\vx}^{(k+1)}-\vx^{(k+1)}\|} &=& \left\|\vx^{(k)}+\frac{b_{i_k} - \va_{i_{k}}^T \vx^{(k)}}{\|\va_{i_{k}}\|^2} \va_{i_k} - \left(\vx^{(k)}+\frac{b_{i_k} + n^{(k)}_{i_k} - \va_{i_{k}}^T \vx^{(k)}}{\|\va_{i_{k}}\|^2} \va_{i_k}\right)\right\|
= \frac{|n^{(k)}_{i_k}|}{\|\va_{i_{k}}\|}.
\end{eqnarray*}
Meanwhile, the quantity $\|\vx^{(k)}-\hat{\vx}^{(k+1)}\|$ is
\begin{eqnarray*}
{\|\vx^{(k)}-\hat{\vx}^{(k+1)}\|} &=& \left\|\vx^{(k)}-\left(\vx^{(k)}+\frac{b_{i_k} - \va_{i_{k}}^T \vx^{(k)}}{\|\va_{i_{k}}\|^2}\va_{i_k}\right)\right\|
= \frac{|b_{i_k} - \va_{i_{k}}^T \vx^{(k)}|}{\|\va_{i_{k}}\|}.
\end{eqnarray*}

Substituting, our equations~\eqref{triangle1} and~\eqref{triangle2} become
\begin{equation}
\|\vx^{(k)}-\vx^*\|^2 = \frac{(b_{i_k} - \va_{i_{k}}^T \vx^{(k)})^2}{\|\va_{i_{k}}\|^2} + \|\hat{\vx}^{(k+1)}-\vx^*\|^2, \label{eq:diagram_eq1}
\end{equation}
and
\begin{equation}
\|\vx^{(k+1)}-\vx^*\|^2 = \|\hat{\vx}^{(k+1)}-\vx^*\|^2 + \frac{(n^{(k)}_{i_k})^2}{\|\va_{i_{k}}\|^2}, \label{eq:diagram_eq2}
\end{equation}
respectively.
The difference of \eqref{eq:diagram_eq1} and \eqref{eq:diagram_eq2} yields
\[
\|\vx^{(k)}-\vx^*\|^2 - \|\vx^{(k+1)}-\vx^*\|^2 = \frac{(b_{i_k} - \va_{i_{k}}^T \vx^{(k)})^2}{\|\va_{i_{k}}\|^2} - \frac{(n^{(k)}_{i_k})^2}{\|\va_{i_{k}}\|^2}.
\]
Now, taking the expected value of both sides with respect to row choice $i_k$, we obtain
\begin{eqnarray*}
\|\vx^{(k)}-\vx^*\|^2-\expect{\|\vx^{(k+1)}-\vx^*\|^2}{i_k} &=&  \expect{\frac{(b_{i_k} - \va_{i_{k}}^T \vx^{(k)})^2}{\|\va_{i_{k}}\|^2}}{i_k} - \expect{\frac{(n^{(k)}_{i_k})^2}{\|\va_{i_{k}}\|^2}}{i_k} \\
{} &=&  \sum_{i=1}^m \frac{(b_i - \va_i^T \vx^{(k)})^2}{\|\va_i\|^2}\frac{\|\va_i\|^2}{\|A\|^2_F} - \sum_{i=1}^m \frac{(n^{(k)}_i)^2}{\|\va_i\|^2}\frac{\|\va_i\|^2}{\|A\|^2_F} \\
{} &=& \frac{\|A\vx^{(k)}-\vb\|^2}{\|A\|^2_F} - \frac{\lVert \vn^{(k)} \rVert_2^2}{\|A\|^2_F} \\
{} &\geq& \frac{\sigma_{min}^2(A)}{\|A\|^2_F} \|\vx^{(k)}-\vx^*\|^2 - \frac{\lVert \vn^{(k)} \rVert_2^2}{\|A\|^2_F},
\end{eqnarray*}
where the second equality applies the row sampling distribution of~\cite{SV09:Randomized-Kaczmarz}, and the inequality follows from residual bounds and the fact that $A$ is full column-rank.  This yields
\begin{eqnarray}
{\expect{\|\vx^{(k+1)}-\vx^*\|^2}{i_k}}
&\leq& \frac{\lVert \vn^{(k)} \rVert_2^2}{\|A\|^2_F} + \left(1 - \frac{\sigma_{min}^2(A)}{\|A\|^2_F}\right) \|\vx^{(k)}-\vx^*\|^2 \\
{} &=& \frac{\lVert \vn^{(k)} \rVert_2^2}{\|A\|^2_F} + \phi \|\vx^{(k)}-\vx^*\|^2. \label{eq:thm2.5line}
\end{eqnarray}

Now, we wish to simplify this recursive expression. We can take the expectation with respect to the noise at the $k$th iteration, $\bm{n^{(k)}}$. We have
\[
\expect{\frac{\lVert \vn^{(k)} \rVert_2^2}{\|A\|^2_F}}{\vn^{(k)}} = \expect{\sum_{i=1}^m \frac{(n^{(k)}_i)^2}{\|A\|^2_F}}{\vn^{(k)}}
= \frac{1}{\|A\|^2_F} \sum_{i=1}^m (s^2+\mu^2) = \frac{m}{\|A\|^2_F}(s^2+\mu^2),
\]
where the second equation follows from the second moment formula.
Therefore we have
\begin{eqnarray*}
{{\expect{\|\vx^{(k+1)}-\vx^*\|^2}{i_k, \vn^{(k)}}}}
&\leq& \frac{m}{\|A\|^2_F}(s^2+\mu^2) + \phi \|\vx^{(k)}-\vx^*\|^2.
\end{eqnarray*}

Repeating this process inductively, we have
\begin{eqnarray*}
    \expect{\| \vx^{(k+1)}-\vx^*\|^2}{} &\leq& {(1+\phi+\phi^2+...+\phi^{k})\frac{m}{\|A\|^2_F}(s^2+\mu^2) + \phi^{k+1} \|\vx_{0}-\vx^*\|^2}\\
    {} &\leq& {\frac{1-\phi^{k+1}}{1-\phi} \frac{m}{\|A\|^2_F}(s^2+\mu^2) + \phi^{k+1} \|\vx_{0}-\vx^*\|^2}.
\end{eqnarray*}
The stated result follows from the independence of all random variables $\{\vn^{(j)}\}_{j=1}^k$ and $\{i_j\}_{j=1}^k$.

\end{proof}

\subsection{Proof of Theorem~\ref{thm:steinerbergerQRKwNoise}}\label{subsec:proof of main}

We now prove our main result Theorem~\ref{thm:steinerbergerQRKwNoise}.  We begin with a lemma which bounds the $q$-quantile of the noisy and corrupted residual by a fraction of the norm of the error and the norm of the noise.

\begin{lemma}\label{lem:quantilebound}
Let $0<q<1-\beta$, let $\vv\in\mathbb{R}^{n}$ be arbitrary, and assume the setup described in Section~\ref{subsec:problem}. Then
\[ q-\text{\normalfont{quant}}\left(\left\{|\langle \vv,\va_i\rangle - b_{i}^{(k)}|\right\}_{i=1}^{m}\right)\leq\frac{1}{m(1-q-\beta)}\bigg(\sqrt{m(1-\beta)}\sigma_{\text{\normalfont{max}}}(A)\lVert \vv - \vx^* \rVert + \|\vn^{(k)}\|_1 \bigg). \]
\end{lemma}

\begin{proof}  We recall that the set of corrupted equations in the $k$th iteration is defined as $C_{k} := \text{supp}(\vc^{(k)}) \subset \{ 1,2,\dots,m \}.$ We begin by considering the rows where there exist no corruption in iteration $k$, that is, for $i \not\in C_{k}$ where
\[\langle \va_i, \vx^* \rangle = b_{i} = b_{i}^{(k)} - n_{i}^{(k)}. \]

\noindent Hence, we have
\begin{eqnarray*}
{| \langle \va_i, \vv \rangle - \langle \va_i, \vx^* \rangle| } = |\langle \va_i, \vv \rangle - b_{i}^{(k)} + n_{i}^{(k)}| \geq |\langle \va_i, \vv \rangle - b_{i}^{(k)}| - |n_{i}^{(k)}|.
\end{eqnarray*}
\noindent This is true for each $i\notin C_{k}$, thus if we take the sum over all these elements we have
\[\sum\limits_{\substack{i=1 \\ i\notin C_{k}}}^{m}| \langle \va_i, \vv \rangle - \langle \va_i, \vx^* \rangle| \geq \sum\limits_{\substack{i=1 \\ i\notin C_{k}}}^{m}|\langle \va_i, \vv \rangle - b_{i}^{(k)}| - \sum\limits_{\substack{i=1 \\ i\notin C_{k}}}^{m}|n_{i}^{(k)}|.\]
\noindent Let us now consider the sum
\[ \sum\limits_{\substack{i=1 \\ i\notin C_{k}}}^{m}|\langle \va_i,\vv\rangle - b_{i}|^2 = \lVert A_{\notin C_{k}}\vv - \vb_{\notin C_{k}} \rVert^2 = \lVert A_{\notin C_{k}}\vv - A_{\notin C_{k}}\vx^* \rVert^2. \]
This can be bounded from above by
\begin{eqnarray*}
{\lVert A_{\notin C_{k}}\vv - A_{\notin C_{k}}\vx^* \rVert^2} \leq \lVert A_{\notin C_{k}} \rVert^2 \dot \lVert \vv - \vx^* \rVert^2 \leq  \lVert A \rVert^2 \dot \lVert \vv - \vx^* \rVert^2 = \sigma_{\text{\normalfont{max}}}(A)^{2}\lVert \vv - \vx^* \rVert^2.
\end{eqnarray*}
Thus we now have that
\[ \sum\limits_{\substack{i=1 \\ i\notin C_{k}}}^{m}|\langle \va_i,\vv\rangle - b_{i}|^2 \leq \sigma_{\text{\normalfont{max}}}^{2}(A)\lVert \vv - \vx^* \rVert^2.\]
Using the $\ell_1 - \ell_2$ norm inequality,
we get
\[\sum\limits_{\substack{i=1 \\ i\notin C_{k}}}^{m}| \langle \va_i, \vv \rangle - \langle \va_i, \vx^* \rangle|  \leq \sqrt{m(1-\beta)} \sigma_{\text{\normalfont{max}}}(A)\lVert \vv - \vx^* \rVert,\]
and hence,
\[\sum\limits_{\substack{i=1 \\ i\notin C_{k}}}^{m}|\langle \va_i, \vv \rangle - b_{i}^{(k)}| - \sum\limits_{\substack{i=1 \\ i\notin C_{k}}}^{m}|n_{i}^{(k)}| \leq \sqrt{m(1-\beta)}\sigma_{\text{\normalfont{max}}}(A)\lVert \vv - \vx^* \rVert.\]
We can rearrange this to arrive at
\begin{eqnarray*}
{\sum\limits_{\substack{i=1 \\ i\notin C_{k}}}^{m}|\langle \va_i, \vv \rangle - b_{i}^{(k)}|} &\leq& \sqrt{m(1-\beta)}\sigma_{\text{\normalfont{max}}}(A)\lVert \vv - \vx^* \rVert + \sum\limits_{\substack{i=1 \\ i\notin C_{k}}}^{m}|n_{i}^{(k)}|\\
{} &\leq&  \sqrt{m(1-\beta)}\sigma_{\text{\normalfont{max}}}(A)\lVert \vv - \vx^* \rVert + \|\vn^{(k)}\|_1.
\end{eqnarray*}

Let $\alpha = q-\text{quant}(\{|\langle \vv,\va_i\rangle - b_{i}^{(k)}|\}_{i=1}^{m})$.
Then at least $(1-q)m$ of the $m$ values $$\left\{\left|\langle \vv, \va_i\rangle - b_{i}^{(k)}\right|\right\}_{i=1}^{m}$$ are at least $\alpha$ and at least $(1-q)m - \beta m$ belong to equations that have not been corrupted. Then
\[ m(1-q-\beta)\alpha\leq\sum_{i\notin C_{k}}|\langle \va_i,\vv \rangle - b_{i}^{(k)}|\leq\sqrt{m(1-\beta)}\sigma_{\text{\normalfont{max}}}(A)\lVert \vv - \vx^* \rVert + \|\vn^{(k)}\|_1. \]
and therefore
\[ q-\text{\normalfont{quant}}\left(\left\{\left|\langle \vv,\va_i\rangle - b_{i}^{(k)}\right|\right\}_{i=1}^{m}\right)\leq\frac{\sqrt{m(1-\beta)}\sigma_{\text{\normalfont{max}}}(A)}{m(1-q-\beta)}\lVert \vv - \vx^* \rVert + \frac{1}{m(1-q-\beta)}\|\vn^{(k)}\|_1. \]
\end{proof}

 Let us now introduce the subset $S_{k} \subset C_{k}$ of corrupted equations which is defined as
\[  S_{k} = \{i \in C_{k}:|\langle \vx^{(k)},\va_i \rangle - b_{i}^{(k)}|\leq q-\text{\normalfont{quant}}(\{|\langle \vx^{(k)},\va_j\rangle - b_{j}^{(k)}|\}_{j=1}^{m})\}.\]
$S_{k}$ is the set of all corrupted equations which end up in the set of $q m$ equations that are being considered for projection by QRK. If $S_{k} = \emptyset$, then when computing $\vx^{(k+1)}$ we can be certain that we are projecting onto an uncorrupted row and thus we would not need to consider the lemma that will follow shortly. Let us then assume, for this lemma, that $|S_{k}|\geq 1$. Also note that $|S_{k}|\leq|C_{k}|\leq|\beta m|$.

\begin{lemma}\label{lem:corruptProjection}
Let $\vx^{(k)}$ denote the iterates of Algorithm~\ref{QuantileRK1} or Algorithm~\ref{QuantileRK2} applied with quantile $q$ to the system defined by $A$ and $ \vb^{(k)} = \vb + \vn^{(k)} + \vc^{(k)}$, in the $k$th iteration. Assuming the setup described in Section~\ref{subsec:problem}, Assumptions~\ref{ass:unit_rows} and~\ref{ass:pos_rate}, $\beta<q<1-\beta$ arbitrary, and $|S_{k}| \geq 1$, then we have
\begin{eqnarray*}
\mathbb{E}_{i\in S_{k}}\rVert \vx^{(k+1)}-\vx^*\lVert^2&\leq&\Bigg(1+\frac{2}{\sqrt{|S_{k}|}}\frac{\sqrt{m(1-\beta)}\sigma_{\text{\normalfont{max}}}^2(A)}{m(1-q-\beta)} +\frac{m(1-\beta)\sigma_{\text{\normalfont{max}}}^2(A)}{m^2 (1-q-\beta)^2}\Bigg)\lVert \vx^{(k)}-\vx^*\rVert^2\\
{} && \quad+ \Bigg(\frac{2}{\sqrt{|S_{k}|}}\frac{\sigma_{\text{\normalfont{max}}}(A)}{m(1-q-\beta)} +\frac{2\sqrt{m(1-\beta)}\sigma_{\text{\normalfont{max}}}(A)}{m^2 (1-q-\beta)^2}\Bigg)\lVert \vx^{(k)}-\vx^*\rVert \|\vn^{(k)}\|_1\\
{} && \quad+  \frac{\|\vn^{(k)}\|_1^2}{m^2(1-q-\beta)^2}.
\end{eqnarray*}
\end{lemma}
\begin{proof}
For any arbitrary vector $\vv\in \mathbb{R}^{n}$,
\[ \lVert \vx^{(k)}+\vv-\vx^*\rVert^2=\lVert \vx^{(k)}-\vx^*\rVert^2 +2\langle \vx^{(k)}-\vx^*,\vv\rangle +\lVert \vv\rVert^2\]
and we will apply this to the special choice
\[ \vv=(b_{i}^{(k)}-\langle \vx^{(k)},\va_i\rangle) \va_i \text{ \normalfont{ where } } i\in S_k.\]
We first observe that, for $i\in S_{k}$, the term $\lVert \vv \rVert^2$ is uniformly small since
\begin{eqnarray*}
{\lVert \vv \rVert^2} &=& \lVert (b_{i}^{(k)}-\langle \vx^{(k)},\va_i\rangle) \va_i \rVert^2 = |b_{i}^{(k)}-\langle \vx^{(k)},\va_i\rangle |^2 \\
{} &\leq&  q-\text{\normalfont{quant}}(\{|\langle \vx^{(k)},\va_j\rangle - b_{j}^{(k)}|^2\}_{j=1}^{m})^2\\
{} &\leq& \frac{1}{m^2(1-q-\beta)^2}\bigg(\sqrt{m(1-\beta)}\sigma_{\text{\normalfont{max}}}(A)\lVert \vx^{(k)} - \vx^* \rVert + \|\vn^{(k)}\|_1\bigg)^2,
\end{eqnarray*}
where the second equation follows from Assumption~\ref{ass:unit_rows} and the last inequality follows from Lemma~\ref{lem:quantilebound}.
It remains to bound $\mathbb{E}_{i\in S_{k}}2\langle \vx^{(k)} - \vx^*,\vv\rangle$. We begin by showing
\begin{eqnarray*}
\mathbb{E}_{i\in S_{k}}2\langle \vx^{(k)}-\vx^*,\vv\rangle &=& \frac{2}{|S_{k}|}\sum_{i\in S_{k}}\langle \vx^{(k)}-\vx^*,(b_{i}^{(k)}-\langle \vx^{(k)},\va_i\rangle)\va_i\rangle \\
{} &=& \frac{2}{|S_{k}|}\sum_{i\in S_{k}}(b_{i}^{(k)}-\langle \vx^{(k)},\va_i\rangle)\langle \vx^{(k)}-\vx^*,\va_i\rangle. \\
\end{eqnarray*}
The Cauchy-Schwarz inequality yields
\[\sum_{i\in S_{k}}(b_{i}^{(k)}-\langle \vx^{(k)},\va_i\rangle)\langle \vx^{(k)}-\vx^*,\va_i\rangle \leq \Bigg(\sum_{i\in S_{k}}(b_{i}^{(k)}-\langle \vx^{(k)},\va_i\rangle)^2 \sum_{i\in S_{k}}\langle \vx^{(k)}-\vx^*,\va_i\rangle^2\Bigg)^{\frac{1}{2}}\]
and thus,
\begin{eqnarray*}
\mathbb{E}_{i\in S_{k}}2\langle \vx^{(k)}-\vx^*,\vv\rangle &\leq& \frac{2}{|S_{k}|}\Bigg(\sum_{i\in S_{k}}(b_{i}^{(k)}-\langle \vx^{(k)},\va_i\rangle)^2 \sum_{i\in S_{k}}\langle \vx^{(k)}-\vx^*,\va_i\rangle^2\Bigg)^{\frac{1}{2}} \\
{} &\leq& \frac{2}{|S_{k}|}\frac{\sqrt{|S_{k}|}}{m(1-q-\beta)}\bigg(\sqrt{m(1-\beta)}\sigma_{\text{\normalfont{max}}}(A)\lVert \vx^{(k)} - \vx^* \rVert + \|\vn^{(k)}\|_1\bigg)\bigg(\sum_{i\in S_{k}}\langle \vx^{(k)}-\vx^*,\va_i\rangle^2\bigg)^{\frac{1}{2}},
\end{eqnarray*}
where the second inequality follows by applying the bound on $\|\vv\|^2$ above.
At this point, we estimate
\[ \sum_{i\in S_{k}}\langle \vx^{(k)}-\vx^*,\va_i\rangle^2\leq \sum_{i=1}^{m} \langle \vx^{(k)}-\vx^*,\va_i\rangle^2 = \lVert A(\vx^{(k)}-\vx^*)\rVert^2\leq\sigma_{\text{\normalfont{max}}}^2(A)\lVert \vx^{(k)}-\vx^*\rVert^2 \]
and hence
\[\mathbb{E}_{i\in S_{k}} 2\langle \vx^{(k)}-\vx^*,v\rangle\leq\frac{2}{\sqrt{|S_{k}|}}\frac{\sigma_{\text{\normalfont{max}}}(A)\cdot\lVert \vx^{(k)}-\vx^*\rVert}{m(1-q-\beta)}\bigg(\sqrt{m(1-\beta)}\sigma_{\text{\normalfont{max}}}(A)\lVert \vx^{(k)} - \vx^* \rVert + \|\vn^{(k)}\|_1\bigg). \]
Summing up now shows that
\begin{eqnarray*}
\mathbb{E}_{i\in S_{k}}\rVert \vx^{(k+1)}-\vx^*\lVert^2&\leq&\Bigg(1+\frac{2}{\sqrt{|S_{k}|}}\frac{\sqrt{m(1-\beta)}\sigma_{\text{\normalfont{max}}}^2(A)}{m(1-q-\beta)} +\frac{m(1-\beta)\sigma_{\text{\normalfont{max}}}^2(A)}{m^2 (1-q-\beta)^2}\Bigg)\lVert \vx^{(k)}-\vx^*\rVert^2\\
{} &&\quad + \Bigg(\frac{2}{\sqrt{|S_{k}|}}\frac{\sigma_{\text{\normalfont{max}}}(A)}{m(1-q-\beta)} +\frac{2\sqrt{m(1-\beta)}\sigma_{\text{\normalfont{max}}}(A)}{m^2 (1-q-\beta)^2}\Bigg)\lVert \vx^{(k)}-\vx^*\rVert \|\vn^{(k)}\|_1\\
{} &&\quad + \frac{\|\vn^{(k)}\|_1^2}{m^2(1-q-\beta)^2}.
\end{eqnarray*}
\end{proof}

\noindent Let $\vx^{(k)}$ be fixed and consider the set of rows which whose residual entry is less than the quantile,
\[B_k :=\{1\leq j\leq m:|\langle \vx^{(k)},\va_j\rangle-b_{j}^{(k)}|\leq q-\text{\normalfont{quant}}(\{|\langle \vx^{(k)},\va_i\rangle - b_{i}^{(k)}|\}_{i=1}^{m})\}. \]
The previous lemma handled the subset of corrupted equations in $B_k$, $S_{k}\subset B_k$. Now we consider the subset $B_k \setminus S_{k}$ and the situation in which we project onto an uncorrupted row.

\begin{proposition} \label{lem:noisyProjection}
Under the assumptions given by Theorem~\ref{thm:steinerbergerQRKwNoise},
\[ \mathbb{E}_{i\in B_k \setminus S_{k}}\lVert \vx^{(k+1)}-\vx^*\rVert^2\leq \noisediff{k} + \bigg(1-\frac{\sigma_{q-\beta,\text{\normalfont{min}}}^2(A)}{q m}\bigg)\lVert \vx^{(k)}-\vx^*\rVert^2,\]
\end{proposition}

Note that the proposition above follows from \eqref{eq:thm2.5line} in the proof of Theorem~\ref{thm:noisyRK} applied to the subsystem $A_{B_k\setminus S_{k}} \vx = \vb^{(k)}_{B_k\setminus S_{k}}$.  With this proposition in hand, we may now prove our main result, Theorem~\ref{thm:steinerbergerQRKwNoise}.

\begin{proof}[Proof of Theorem~\ref{thm:steinerbergerQRKwNoise}]
During the application of QRK we encounter three possibilities for the row that we select at each iteration: the row is either in the set of corrupted rows within the set of admissible rows, in the set of non-corrupted rows within the set of admissible rows, or not in the set of admissible rows. Thus, when we are selecting a row from $B_k$, the probabilities to select a corrupted row or an uncorrected row are
\begin{equation}
{\mathbb{P}(i_k\in S_{k} | i_k\in B_k)} = \frac{|S_{k}|}{q m}
\end{equation}
and
\begin{equation}
{\mathbb{P}(i_k\in B_k\setminus S_{k} | i_k\in B_k)} = 1 - \frac{|S_{k}|}{q m}.
\end{equation}

Now, if we are using Algorithm~\ref{QuantileRK2}, we must condition on the event of working wholly within $B_k$. If we are using Algorithm~\ref{QuantileRK1}, we can sample either within or beyond $B_k$. Applying the law of total expectation twice gives us the expression
\begin{eqnarray*}
\expect{\|\vx^{(k+1)} - \vx^*\|^2}{i_k} &=& \bigg(\conexpect{\|\vx^{(k+1)} - \vx^*\|^2}{i_k}{i_k \in S_{k}} \cdot \mathbb{P}(i_k\in S_{k} | i_k\in B_k) \\
{} &&\quad + \conexpect{\|\vx^{(k+1)} - \vx^*\|^2}{i_k}{i_k \in B_k\setminus S_{k}} \cdot \mathbb{P}(i_k\in B_k\setminus S_{k} | i_k\in B_k)\bigg) \cdot \mathbb{P}(i\in B_k)  \\
{} &&\quad + \conexpect{\|\vx^{(k+1)} - \vx^*\|^2}{i_k}{i_k \notin B_k} \cdot \mathbb{P}(i_k\notin B_k)
\end{eqnarray*}
where $\mathbb{P}(i_k\in B_k)$ and $\mathbb{P}(i_k\notin B_k)$ are 1 and 0 respectively for Algorithm~\ref{QuantileRK2}, and have the values $q$ and $1-q$ when using Algorithm~\ref{QuantileRK1}. \\\\
Using the probabilities above with Lemma~\ref{lem:corruptProjection} and Proposition~\ref{lem:noisyProjection}, we arrive at
\begin{eqnarray*}
\expect{\|\vx^{(k+1)} - \vx^*\|^2}{i_k} &\leq& \Bigg(\frac{|S_{k}|}{q m}\Bigg(1+\frac{2}{\sqrt{|S_{k}|}}\frac{\sqrt{m(1-\beta)}\sigma_{\text{\normalfont{max}}}^2(A)}{m(1-q-\beta)} +\frac{m(1-\beta)\sigma_{\text{\normalfont{max}}}^2(A)}{m^2 (1-q-\beta)^2}\Bigg)\lVert \vx^{(k)}-\vx^*\rVert^2\\
{} &&\quad + \frac{|S_{k}|}{q m}\Bigg(\frac{2}{\sqrt{|S_{k}|}}\frac{\sigma_{\text{\normalfont{max}}}(A)}{m(1-q-\beta)} +\frac{2\sqrt{m(1-\beta)}\sigma_{\text{\normalfont{max}}}(A)}{m^2 (1-q-\beta)^2}\Bigg)\lVert \vx^{(k)}-\vx^*\rVert \lVert \vn^{(k)}\rVert\\
{} &&\quad + \frac{|S_{k}|}{q m}\frac{\lVert \vn^{(k)} \rVert^2}{m^2(1-q-\beta)^2} + \bigg( 1 - \frac{|S_{k}|}{q m}\bigg)\bigg(1-\frac{\sigma_{q-\beta,\text{\normalfont{min}}}^2(A)}{q m}\bigg)\lVert \vx^{(k)}-\vx^*\rVert^2\\
{} &&\quad + \bigg( 1 - \frac{|S_{k}|}{q m}\bigg)\noisediff{k}\Bigg) \cdot \mathbb{P}(i_k\in B_k) + \lVert \vx^{(k)}-\vx^*\rVert^2 \cdot \mathbb{P}(i_k\notin B_k)
\end{eqnarray*}
Multiplying in the probabilities and regrouping, we arrive at the expression
\begin{eqnarray*}
\expect{\|\vx^{(k+1)} - \vx^*\|^2}{i_k} &\leq& \Bigg(\bigg(1+\frac{2\sqrt{|S_{k}|}\sqrt{m(1-\beta)}\sigma_{\text{\normalfont{max}}}^2(A)}{q m^2(1-q-\beta)^2}+ \frac{|S_{k}|(1-\beta)\sigma_{\text{\normalfont{max}}}^2(A)}{q m^2 (1-q-\beta)} -\frac{\sigma_{q-\beta,\text{\normalfont{min}}}^2(A)}{q m}\\
{} &&\quad + \frac{|S_{k}|\sigma_{q-\beta,\text{\normalfont{min}}}^2(A)}{q^2 m^2}\bigg)\lVert \vx^{(k)}-\vx^*\rVert^2 \\
{} &&\quad+ \bigg(\frac{2\sqrt{|S_{k}|}\sigma_{\text{\normalfont{max}}}(A)}{q m^2(1-q-\beta)}+\frac{2|S_{k}|\sqrt{m(1-\beta)}\sigma_{\text{\normalfont{max}}}(A)}{q m^3(1-q-\beta)^2} \bigg)\lVert \vx^{(k)}-\vx^*\rVert\lVert \vn^{(k)} \rVert \\
{} &&\quad + \bigg( 1 - \frac{|S_{k}|}{q m}\bigg)\noisediff{k} +\frac{|S_{k}|\lVert \vn^{(k)} \rVert^2}{q m^3(1-q-\beta)^2}\Bigg)\cdot\mathbb{P}(i_k\in B_k) + \lVert \vx^{(k)}-\vx^*\rVert^2 \cdot \mathbb{P}(i_k\notin B_k).
\end{eqnarray*}
Rearranging the expression we arrive at
\begin{eqnarray*}
\expect{\|\vx^{(k+1)} - \vx^*\|^2}{i_k} &\leq& \lVert \vx^{(k)}-\vx^*\rVert^2 + \bigg(\frac{\sigma_{\text{\normalfont{max}}}^2(A)}{q m}\bigg(\frac{2\sqrt{|S_{k}|}\sqrt{m(1-\beta)}}{m(1-q-\beta)}+ \frac{|S_{k}|(1-\beta)}{m(1-q-\beta)^2}\bigg)\\
{} &&\quad - \frac{\sigma_{q-\beta,\text{\normalfont{min}}}^2(A)}{q m}\bigg(1 -\frac{|S_{k}|}{q m}\bigg)\bigg)\cdot\mathbb{P}(i_k\in B_k)\cdot\lVert \vx^{(k)}-\vx^*\rVert^2 \\
{} &&\quad + \Bigg(\Bigg( \frac{\sigma_{\text{\normalfont{max}}}(A)}{q m}\bigg( \frac{2\sqrt{|S_{k}|}}{m(1-q-\beta)}+\frac{2|S_{k}|\sqrt{m(1-\beta)}}{m^2(1-q-\beta)^2} \bigg)\Bigg)\lVert \vx^{(k)}-\vx^*\rVert\lVert \vn^{(k)} \rVert\\
{} &&\quad + \bigg( 1 - \frac{|S_{k}|}{q m}\bigg)\noisediff{k} + \frac{|S_{k}|\lVert \vn^{(k)} \rVert^2}{q m^3(1-q-\beta)^2}\Bigg)\cdot\mathbb{P}(i_k\in B_k).
\end{eqnarray*}

Let us consider the expression
\[\frac{\sigma_{\text{\normalfont{max}}}^2(A)}{q m}\bigg(\frac{2\sqrt{|S_{k}|}\sqrt{m(1-\beta)}}{m(1-q-\beta)}+ \frac{|S_{k}|(1-\beta)}{m(1-q-\beta)^2}\bigg) -\frac{\sigma_{q-\beta,\text{\normalfont{min}}}^2(A)}{q m}\bigg(1 -\frac{|S_{k}|}{q m}\bigg).\]
It is the case that it is monotonically increasing in $|S_{k}|$, thus let us use the worst case $|S_{k}|=\beta m$ to produce the upper bound
\begin{equation} \label{worst_case_expression}
    \frac{\sigma_{\text{\normalfont{max}}}^2(A)}{q m}\bigg(\frac{2\sqrt{\beta(1-\beta)}}{(1-q-\beta)}+ \frac{\beta(1-\beta)}{(1-q-\beta)^2}\bigg) -\frac{\sigma_{q-\beta,\text{\normalfont{min}}}^2(A)}{q m}\bigg(\frac{q-\beta}{q}\bigg).
\end{equation}
We denote the negation of this upper bound $\vartheta$ and have
\begin{eqnarray*}
\expect{\|\vx^{(k+1)} - \vx^*\|^2}{i_k} &\leq& \lVert \vx^{(k)}-\vx^*\rVert^2 - \vartheta \lVert \vx^{(k)}-\vx^*\rVert^2 \cdot \mathbb{P}(i_k\in B_k)\\
{} &&\quad + \Bigg(\Bigg( \frac{\sigma_{\text{\normalfont{max}}}(A)}{q m}\bigg( \frac{2\sqrt{|S_{k}|}}{m(1-q-\beta)}+\frac{2|S_{k}|\sqrt{m(1-\beta)}}{m^2(1-q-\beta)^2} \bigg)\Bigg)\lVert \vx^{(k)}-\vx^*\rVert\lVert \vn^{(k)} \rVert\\
{} &&\quad +  \bigg( 1 - \frac{|S_{k}|}{q m}\bigg)\noisediff{k} + \frac{|S_{k}|\lVert \vn^{(k)} \rVert^2}{q m^3(1-q-\beta)^2}\Bigg)\cdot\mathbb{P}(i_k\in B_k)
\end{eqnarray*}
Here, we plug in the worst case values of $|S_{k}|$ for each term, that is the value for which these terms are maximized, giving us
\begin{eqnarray*}
\expect{\|\vx^{(k+1)} - \vx^*\|^2}{i_k} &\leq& \lVert \vx^{(k)}-\vx^*\rVert^2 - \vartheta \lVert \vx^{(k)}-\vx^*\rVert^2 \cdot \mathbb{P}(i_k\in B_k)\\
{} &&\quad + \Bigg(\Bigg( \frac{\sigma_{\text{\normalfont{max}}}(A)}{q m}\bigg( \frac{2\sqrt{\beta m}}{m(1-q-\beta)}+\frac{2\beta\sqrt{m(1-\beta)}}{m(1-q-\beta)^2} \bigg)\Bigg)\lVert \vx^{(k)}-\vx^*\rVert\lVert \vn^{(k)} \rVert\\
{} &&\quad +  \noisediff{k} + \frac{\beta\lVert \vn^{(k)} \rVert^2}{q m^2(1-q-\beta)^2}\Bigg)\cdot\mathbb{P}(i_k\in B_k).
\end{eqnarray*}

Now, applying the inequality $\|\vx^{(k)} - \vx^*\| \|\vn^{(k)}\| \le \frac12\left(\|\vx^{(k)} - \vx^*\|^2 +  \|\vn^{(k)}\|^2\right)$, we have
\begin{eqnarray*}
\expect{\|\vx^{(k+1)} - \vx^*\|^2}{i_k} &\leq& \lVert \vx^{(k)}-\vx^*\rVert^2 - \vartheta \lVert \vx^{(k)}-\vx^*\rVert^2 \cdot \mathbb{P}(i\in B_k)\\
{} &&\quad + \frac{\sigma_{\text{\normalfont{max}}}(A)}{q m}\bigg( \frac{\sqrt{\beta m}}{m(1-q-\beta)}+\frac{\beta\sqrt{m(1-\beta)}}{m(1-q-\beta)^2} \bigg)\lVert \vx^{(k)}-\vx\rVert^2\cdot\mathbb{P}(i_k\in B_k)\\
{} &&\quad +  \Bigg(\Bigg( \frac{\sigma_{\text{\normalfont{max}}}(A)}{q m}\bigg( \frac{\sqrt{\beta m}}{m(1-q-\beta)}+\frac{\beta\sqrt{m(1-\beta)}}{m(1-q-\beta)^2} \bigg) +\frac{\beta}{q m^2(1-q-\beta)^2}\Bigg)\lVert \vn^{(k)} \rVert^2\\
{} &&\quad +  \noisediff{k}\Bigg)\cdot\mathbb{P}(i_k\in B_k).
\end{eqnarray*}
Now note that
\[\varphi = \vartheta - \frac{\sigma_{\text{\normalfont{max}}}(A)}{q m}\bigg( \frac{\sqrt{\beta m}}{m(1-q-\beta)}+\frac{\beta\sqrt{m(1-\beta)}}{m(1-q-\beta)^2} \bigg) \]
and
\[\sysconst = \frac{\sigma_{\text{\normalfont{max}}}(A)}{q m}\bigg( \frac{\sqrt{\beta m}}{m(1-q-\beta)}+\frac{\beta\sqrt{m(1-\beta)}}{m(1-q-\beta)^2} \bigg) +\frac{\beta}{q m^2(1-q-\beta)^2}.\]
Thus
\[\expect{\|\vx^{(k+1)} - \vx^*\|^2}{i_k} \leq \bigg(1 - \varphi\cdot\mathbb{P}(i_k\in B_k)\bigg)\lVert \vx^{(k)}-\vx^*\rVert^2 + \bigg(\noisediff{k} + \sysconst \noisediffb{k}\bigg)\cdot\mathbb{P}(i_k\in B_k), \]
and letting $p = \mathbb{P}(i_k\in B_k)$,
\begin{align}
    \expect{\|\vx^{(k+1)} - \vx^*\|^2}{i_k} &\leq \bigg(1 - p\varphi\bigg)\lVert \vx^{(k)}-\vx^*\rVert^2 + p\bigg(\noisediff{k} + \sysconst \noisediffb{k}\bigg) \label{eq:arbitrary noise theorem}
    \\&\leq \bigg(1 - p\varphi\bigg)\lVert \vx^{(k)}-\vx^*\rVert^2 + p\bigg(\frac{\|\vn^{(k)}\|_2^2}{(q-\beta)m} + \sysconst \noisediffb{k}\bigg)
\end{align}
since $|B_k\setminus S_k| \ge (q-\beta)m$ and the rows of $A$ are normalized under Assumption~\ref{ass:unit_rows}.

We recursively apply this bound to bound the total expectation with respect to all choices of row indices in all iterations, which gives
\[
\expect{\lVert \vx^{(k+1)}-\vx^*\rVert^2}{} \leq (1 - p\varphi)^{k+1}\lVert\vx^{(0)}-\vx^*\rVert^2 + p\sum_{j=0}^{k} (1 - p\varphi)^{k-j} \left(\frac{\|\vn^{(k)}\|_2^2}{(q-\beta)m} + \sysconst \noisediffb{j}\right).
\]
Recall \[\horizon{k} := \sum_{j=0}^{k} (1 - p\varphi)^{k-j} \left(\frac{\|\vn^{(k)}\|_2^2}{(q-\beta)m} + \sysconst \noisediffb{j}\right)\]
\noindent and therefore
\[
\expect{\lVert \vx^{(k+1)}-\vx^*\rVert^2}{} \leq (1 - p\varphi)^{k+1}\lVert\vx^{(0)}-\vx^*\rVert^2 + p\horizon{k}.
\]
\end{proof}

\subsection{Proofs of Corollaries}\label{subsec:proofs of corollaries}

In this section, we prove the corollaries of our main result, Theorem~\ref{thm:steinerbergerQRKwNoise}.  Our first corollary case, Corollary~\ref{cor:single_noisy_cor}~\ref{cor:boundednoisecorr}, specializes Theorem~\ref{thm:steinerbergerQRKwNoise} to the case that the noise in each iteration is uniformly bounded.
\begin{proof}[Proof of Corollary~\ref{cor:single_noisy_cor}~\ref{cor:boundednoisecorr}]
Under Assumption~\ref{ass:unit_rows}, we have $\noisediff{j} \le n_{\max}^2$.  Additionally, we have $\noisediffb{j} \le m^2 n_{\max}^2$.
Hence, introducing these bounds into~\eqref{eq:arbitrary noise theorem} and iterating as in the proof of Theorem~\ref{thm:steinerbergerQRKwNoise}, we have the following bound on the convergence horizon term,
\begin{equation}
     (1 + \sysconst m^2)(n_{max})^2 \sum_{j=0}^{k} (1 - \qrkprob \varphi)^{k-j} \le (1 + \sysconst m^2)(n_{max})^2 \frac{1-(1 - \qrkprob \varphi)^{k+1}}{1-(1 - \qrkprob \varphi)}.
\end{equation}
Therefore, we get
\begin{eqnarray*}
    \expect{\| \vx^{(k+1)}-\vx^*\|^2}{i_{1}, \cdots, i_{k}} &\leq& {(1 - \qrkprob \varphi)^{k+1}\lVert\vx^{(0)}-\vx^*\rVert^2 + \qrkprob(1 + \sysconst m^2)(n_{max})^2 \frac{1-(1 - \qrkprob \varphi)^{k+1}}{1-(1 - \qrkprob \varphi)}} \\
    {} &\leq& {(1 - \qrkprob \varphi)^{k+1}\lVert\vx^{(0)}-\vx^*\rVert^2 + (1 + \sysconst m^2)(n_{max})^2 \frac{1-(1 - \qrkprob \varphi)^{k+1}}{\varphi}}.
\end{eqnarray*}
\end{proof}

Our second corollary case, Corollary~\ref{cor:single_noisy_cor}~\ref{cor:defnoise}, specializes Theorem~\ref{thm:steinerbergerQRKwNoise} to the case that the noise is sampled from a known distribution.
\begin{proof}[Proof of Corollary~\ref{cor:single_noisy_cor}~\ref{cor:defnoise}]
Continuing the proof of Theorem~\ref{thm:steinerbergerQRKwNoise} from equation \eqref{eq:arbitrary noise theorem}, we have
\[\expect{\| \vx^{(k+1)}-\vx^*\|^2}{i_k} \leq (1-\qrkprob\varphi)\lVert \vx^{(k)}-\vx^*\rVert^2 + \qrkprob\left(\noisediff{k} + \sysconst\noisediffb{k}\right). \]
We next take the expected value of the inequality with respect to the noise term at the $k^{\text{th}}$ iteration,
\begin{eqnarray*}
{\expect{\expect{\| \vx^{(k+1)}-\vx^*\|^2}{i_{k}}}{\vn^{(k)}}} &\leq& {(1-\qrkprob\varphi)\lVert\vx^{(k)}-\vx^*\rVert^2 + \qrkprob\left(\expect{\noisediff{k}}{\vn^{(k)}} + \sysconst \expect{\noisediffb{k}}{\vn^{(k)}}\right)}.\\
\end{eqnarray*}
We now apply the second moment formula and Assumption~\ref{ass:unit_rows} to obtain
\begin{equation}
    {\expect{\noisediff{k}}{\vn^{(k)}}} = {(s^2 + \mu^2) \sum_{i\in B\setminus S_{k}} \frac{1}{|{B\setminus S_{k}}|}} = {s^2 + \mu^2}.
\end{equation}
Again, applying the second moment formula yields
\begin{equation}
    {\expect{\noisediffb{k}}{\vn^{(k)}}} = {\left(\sum\limits_{i=1}^{m} \expect{ |n_i^{(k)}|}{\vn^{(k)}}\right)^2 + \sum\limits_{i=1}^{m} \text{Var} \left[|n_i^{(k)}|\right]} = {m^2(\mu')^2 + m(s')^2.}
\end{equation}
Thus
\begin{eqnarray*}
\expect{\expect{\| \vx^{(k+1)}-\vx^*\|^2}{i_{k}}}{\vn^{(k)}} &\leq& {(1-\qrkprob\varphi)\lVert\vx^{(k)}-\vx^*\rVert^2 + p({\mu^2 + s^2} + \sysconst [m^2(\mu')^2 + m(s')^2]).}
\end{eqnarray*}
Let
\[
\distfunc:= {\mu^2 + s^2} + \sysconst [m^2(\mu')^2 + m(s')^2].
\]
Therefore we have
\[
\expect{\expect{\| \vx^{(k+1)}-\vx^*\|^2}{i_{k}}}{\vn^{(k)}} \leq (1-\qrkprob\varphi)\lVert\vx^{(k)}-\vx^*\rVert^2 + \qrkprob\distfunc .
\]
Inducting yields
\begin{eqnarray*}
\expect{\| \vx^{(k+1)}-\vx^*\|^2}{} &\leq& {(1-\qrkprob \varphi)^{k+1}\lVert\vx^{(0)}-\vx^*\rVert^2 + \frac{1-(1-\qrkprob \varphi)^{k+1}}{1-(1-\qrkprob \varphi)} \qrkprob \distfunc}\\
{} &=& {(1-\qrkprob \varphi)^{k+1}\lVert\vx^{(0)}-\vx^*\rVert^2 + \frac{1-(1-\qrkprob \varphi)^{k+1}}{\varphi} \distfunc.}
\end{eqnarray*}
\end{proof}

Our last corollary case, Corollary~\ref{cor:single_noisy_cor}~\ref{cor:gaussian noise}, specializes to the case that the noise has entries sampled from a Gaussian distribution.
\begin{proof}[Proof of Corollary~\ref{cor:single_noisy_cor}~\ref{cor:gaussian noise}]
    Applying Corollary~\ref{cor:single_noisy_cor}~\ref{cor:defnoise}, we know that for an arbitrary distribution of $n_i^{(k)}$ with mean $\mu$ and standard deviation $s$, and distribution of $|n_i^{(k)}|$ with mean $\mu'$ and standard deviation $s'$, we have
    \[
    \expect{\| \vx^{(k+1)}-\vx^*\|^2}{} \leq (1-\qrkprob \varphi)^{k+1}\lVert\vx^{(0)}-\vx^*\rVert^2 + \frac{1-(1-\qrkprob \varphi)^{k+1}}{\varphi} \distfunc.
    \]
    Under the assumption that $n_i^{(k)}$ is sampled i.i.d.\ from $\mathcal{N}(0,s^2)$, we have $\mu' = s\sqrt{\frac{2}{\pi}}$ and $s' = s\sqrt{1-\frac{2}{\pi}}$; see e.g.,~\cite{leone1961folded}. Hence,
    \begin{eqnarray*}
        {\frac{1-(1-\qrkprob \varphi)^{k+1}}{\varphi} \distfunc} &=& {\frac{1-(1-\qrkprob \varphi)^{k+1}}{\varphi}s^2\left(1 + \sysconst \left[m^2 \frac{2}{\pi} + m \left(1-\frac{2}{\pi}\right)\right]\right)}
    \end{eqnarray*}
    and thus
    \[
    \expect{\| \vx^{(k+1)}-\vx^*\|^2}{} \leq (1-\qrkprob \varphi)^{k+1}\lVert\vx^{(0)}-\vx^*\rVert^2 + \frac{1-(1-\qrkprob \varphi)^{k+1}}{\varphi}s^2\left(1 + \sysconst \left[m^2 \frac{2}{\pi} + m \left(1-\frac{2}{\pi}\right)\right]\right).
    \]
\end{proof}

Finally, our last corollary, Corollary~\ref{cor:noisy corruption detection}, uses Markov's inequality to transform the bound on the expectation provided by Theorem~\ref{thm:steinerbergerQRKwNoise} and Corollary~\ref{cor:single_noisy_cor} to a lower bound on the probability that the time-varying corruption can be detected by examining the largest entries in the residual at the $k$th iteration.
\begin{proof}[Proof of Corollary~\ref{cor:noisy corruption detection}]
By Corollary~\ref{cor:single_noisy_cor}~\ref{cor:boundednoisecorr},
\[\expect{\| \vx^{(k)}-\vx^*\|^2}{} \leq (1 - \qrkprob \varphi)^{k}\lVert\vx^{(0)}-\vx^*\rVert^2 + (1 + \sysconst m^2) \frac{n_{max}^2}{\varphi} \]
and applying Markov's inequality yields
\[\mathbb{P}[\lVert \vx^{(k)}-\vx^*\rVert^2 \ge a] \le \frac{(1 - \qrkprob \varphi)^{k}\lVert\vx^{(0)}-\vx^*\rVert^2 + (1 + \sysconst m^2) \frac{n_{max}^2}{\varphi}}{a}\] for all $a > 0$. If we choose $a = (c^{(k)}_{\min})^2/4$, then
\[\mathbb{P}[\lVert \vx^{(k)}-\vx^*\rVert \le \frac{c^{(k)}_{\min}}{2}] \ge 1 - \frac{4(1 - p\varphi)^{k}\lVert\vx^{(0)}-\vx^*\rVert^2}{(c^{(k)}_{\min})^2} - \frac{1}{M},\] since $(1 +  \sysconst m^2)\frac{n_{\max}^2}{\varphi}  \le (c_{\min}^{(k)})^2/4M$.

Now, note that in the case that $\lVert \vx^{(k)}-\vx^*\rVert \le \frac{c^{(k)}_{\min}}{2},$ if $i \not\in C_{k}$, we have
\begin{align*}
    |\va_i^\top \vx^{(k)} - b_i^{(k)}| = |\va_i^\top \vx^{(k)} - b_i| \le \|\vx^{(k)} - \vx^*\| \le \frac{c^{(k)}_{\min}}{2},
\end{align*}
while if $i \in C_{k}$, we have
\begin{align*}
    \left|\va_i^\top \vx^{(k)} - b_i^{(k)}\right| = \left|\va_i^\top \vx^{(k)} - b_i - c_i^{(k)}\right| \ge \left|\left|c_i^{(k)}\right| - \left|\va_i^\top \vx^{(k)} - b_i\right|\right| \ge c^{(k)}_{\min} - \frac{c^{(k)}_{\min}}{2}.
\end{align*}
Thus, we have \[\text{supp}(\vc^{(k)}) \subset \text{argmax}_{\substack{S \subset [m], \\|S| = \lfloor\beta m\rfloor}} \sum_{i \in S} |\va_i^\top \vx^{(k)} - b_i^{(k)}|.\]
\end{proof}

\section{Numerical Experiments}\label{sec:numerical experiments}
In this section, we present numerical experiments using QRK with various quantiles $q$. In the simulated experiments, we compare the theoretical convergence guarantee, Theorem~\ref{thm:steinerbergerQRKwNoise}, to the empirical performance of QRK,
measured by approximation error $\|\vx^{(k)} - \vx^*\|^2$, averaged over 10 random trials.  Note that in our calculation of the upper bound provided in Theorem~\ref{thm:steinerbergerQRKwNoise}, we will not directly calculate $\sigma_{q-\beta,\min}(A)$, but rather use the lower bound reported by Steinerberger in~\cite{steinerberger2021quantile}.

The number of rows $m = 20000$ and number of columns $n = 100$ are fixed
for all simulated experiments. The solution to the system is a vector $\vx^* \in \mathbb{R}^n$
where each entry is drawn i.i.d.\ from a standard Gaussian distribution. In
each experiment, the systems are built as $\vb = A\vx^*$ with $\vn^{(k)}$ and $\vc^{(k)}$ generated as described below.

The experiments presented in this section were performed in Python 3.9.7 with Numpy 1.23.5 on a Linux server with 1TiB memory and an AMD EPYC 7313 16-core processor.
In all experiments, we use the Algorithm~\ref{QuantileRK2} implementation of QRK and the corresponding theoretical bounds.  Results for Algorithm~\ref{QuantileRK1} would be similar.

\subsection{QRK on systems with static vs. time-varying corruption and noise}

In this section, we explore the effect of time-varying noise and corruption on the convergence of QRK.  We generate a system with $A \in \mathbb{R}^{20000 \times 100}$ generated with i.i.d.\ entries sampled from $\mathcal{N}(0,1)$ and row-normalize $A$, we generate $\vx^* \in \mathbb{R}^{100}$ with i.i.d.\ entries sampled from $\mathcal{N}(0,1)$.  For the static noise-free system, we generate $\vb^{(k)} = A\vx^* + \vc$ for a static corruption vector $\vc$ with $\beta = |\text{supp}(\vc)|/m = 10^{-3}$.  For the varying corruption and noise-free system, we generate $\vb^{(k)} = A\vx^* + \vc^{(k)}$ for time-varying corruption vectors $\vc^{(k)}$ with $\beta = |\text{supp}(\vc^{(k)})|/m = 10^{-3}$ in each iteration.  The positions of all corruption vectors are randomly selected and these entries are set to $10$.  We apply QRK with both these systems (static and time-varying) and compare the empirical convergence to the bound in Theorem~\ref{thm:steinerbergerQRKwNoise}; see the left plot of Figure~\ref{fig:static-time-varying-corruption}.  For the static noisy system, we generate $\vb^{(k)} = A\vx^* + \vc + \vn$ for a static corruption vector $\vc$ with $\beta = |\text{supp}(\vc)|/m = 10^{-3}$ and a static noise vector $\vn$ with entries sampled i.i.d.\ from $\mathcal{N}(0,0.001)$.  For the varying corruption and varying noise system, we generate $\vb^{(k)} = A\vx^* + \vc^{(k)} + \vn^{(k)}$ for time-varying corruption vectors $\vc^{(k)}$ with $\beta = |\text{supp}(\vc^{(k)})|/m = 10^{-3}$ and time-varying noise vectors $\vn^{(k)}$ with entries sampled i.i.d.\ from $\mathcal{N}(0,0.001)$, in each iteration.  The positions of all corruption vectors are randomly selected and these entries are set to $10$.  We apply QRK with both these systems (static and time-varying) and compare the empirical convergence to the bound in Theorem~\ref{thm:steinerbergerQRKwNoise}; see the right plot of Figure~\ref{fig:static-time-varying-corruption}.  We note that the time-varying aspect of the corruption and noise has little effect on the behavior of QRK or the bound provided.

\begin{figure}
    \centering
    \includegraphics[width=0.49\textwidth]{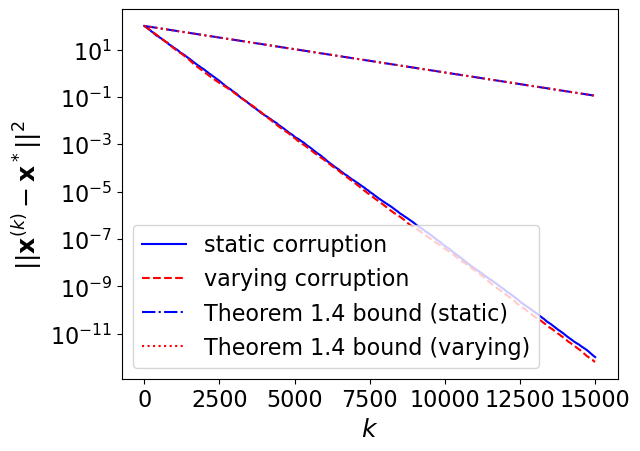}
    \includegraphics[width=0.49\textwidth]{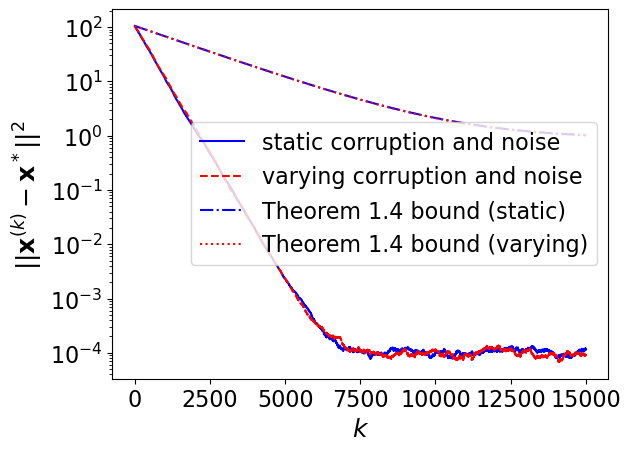}
    \caption{Empirical behavior of QRK ($q = 0.6$) and upper bound provided in Theorem~\ref{thm:steinerbergerQRKwNoise} on (left) static noise-free ($\vb^{(k)} = \vb + \vc$) and time-varying noise-free ($\vb^{(k)} = \vb + \vc^{(k)}$) systems and (right) static noisy ($\vb^{(k)} = \vb + \vc + \vn$) and time-varying noisy ($\vb^{(k)} = \vb + \vc^{(k)} + \vn^{(k)}$) systems defined by $A \in \mathbb{R}^{20000 \times 100}$ where corruption vectors are generated with $\beta m$ uniformly randomly selected entries equal to one and the remainder equal to 0 where $\beta = 0.001$. The noise vectors have entries sampled i.i.d.\ from $\mathcal{N}(0,0.001)$. }
    \label{fig:static-time-varying-corruption}
\end{figure}

\subsection{QRK on systems with time-varying noise and corruption}
In this section we compare the upper bound in Theorem~\ref{thm:steinerbergerQRKwNoise} to the empirical behavior of QRK on a system with time-varying noise and corruption.
We generate a system with $A \in \mathbb{R}^{20000 \times 100}$ generated with i.i.d.\ entries sampled from $\mathcal{N}(0,1)$ and row-normalize $A$, we generate $\vx^* \in \mathbb{R}^{100}$ with i.i.d.\ entries sampled from $\mathcal{N}(0,1)$, and $\vb^{(k)} = A\vx^* + \vn^{(k)} + \vc^{(k)}$ where in each iteration $\vn^{(k)}$ is sampled from $\mathcal{N}(0,s^2)$ and we have corruption vector $\vc^{(k)}$ defined by $\beta = |\text{supp}(\vc^{(k)})|/m$ with randomly selected entries set to $10$. Finally, we implement QRK on this system with a quantile of size $q$. In Figure~\ref{fig:thm1.4q} we provide a set of three plots illustrating how the empirical behavior of QRK and the bound of Theorem~\ref{thm:steinerbergerQRKwNoise} vary with respect to $q$, $\beta$, and $s$.

\begin{figure}
    \includegraphics[width=0.5\textwidth]{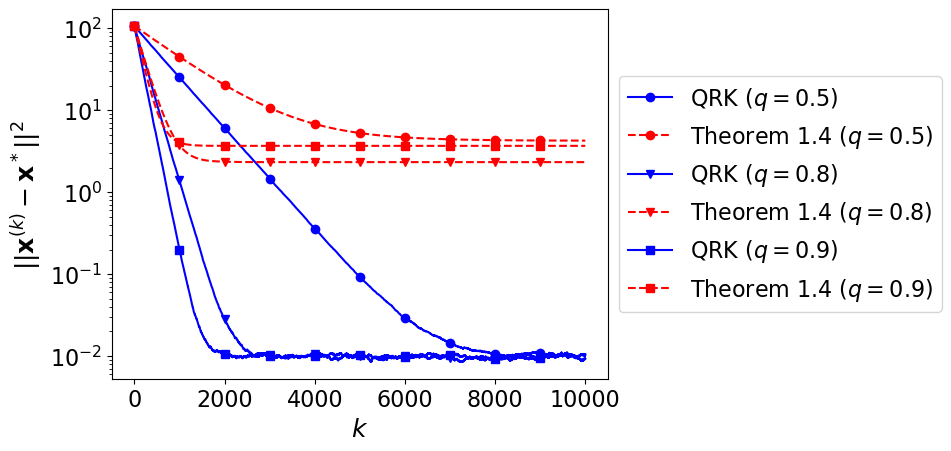}\hfill%
    \includegraphics[width=0.5\textwidth]{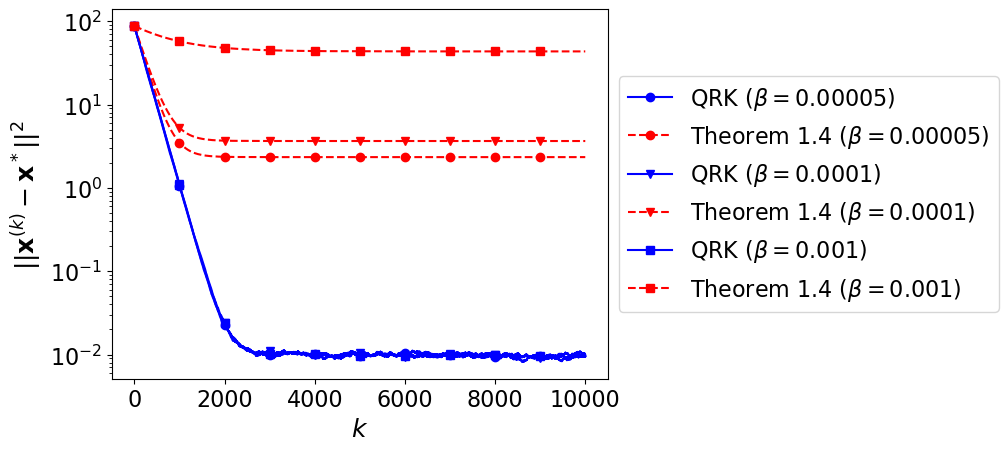}\hfill%
    \includegraphics[width=0.5\textwidth]{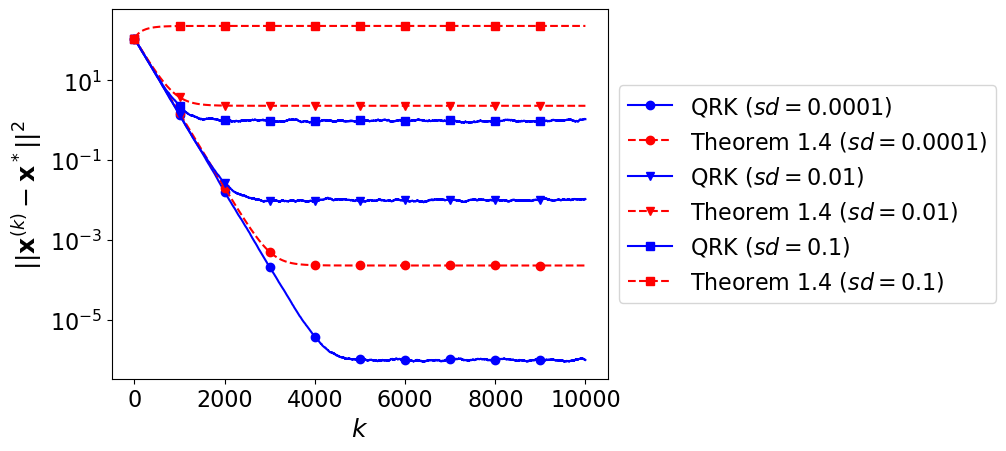}
    \caption{Empirical behavior of QRK and upper bound provided in Theorem~\ref{thm:steinerbergerQRKwNoise} on system defined by $A \in \mathbb{R}^{20000 \times 100}$ and $\vb^{(k)} = \vb + \vn^{(k)} + \vc^{(k)}$ where corruption $\vc^{(k)}$ is a vector with $\beta m$ uniformly randomly selected entries equal to the corruption size (10) and the remainder equal to 0. The noise $\vn^{(k)}$ is an i.i.d.\ random vector sampled in each iteration with entries from $\mathcal{N}(0,s^2)$. The upper left shows behavior as $q$ varies between 0.5, 0.8 and 0.9, while $\beta = 0.00005$ and $s = 0.01$. The upper right shows behavior as $\beta$ varies between 0.00005, 0.0001 and 0.001, while $q = 0.8$ and $s = 0.01$. The lower center shows behavior as $s$ varies between 0.0001, 0.01 and 0.1, while $q = 0.8$ and $\beta = 0.00005$. Experiments are averaged over 10 trials.}
    \label{fig:thm1.4q}
\end{figure}

\subsection{QRK on systems with significant corruption}

In this section, we examine the robustness of QRK to systems with larger fraction of corruption, even for systems in which our theoretical convergence results do not apply.  We generate a system with $A \in \mathbb{R}^{20000 \times 100}$ generated with i.i.d.\ entries sampled from $\mathcal{N}(0,1)$ and row-normalize $A$, we generate $\vx^* \in \mathbb{R}^{100}$ with i.i.d.\ entries sampled from $\mathcal{N}(0,1)$, and $\vb^{(k)} = A\vx^* + \vn^{(k)} + \vc^{(k)}$ where in each iteration $\vn^{(k)}$ is sampled from $\mathcal{N}(0,10^{-8})$ and we have corruption vector $\vc^{(k)}$ defined by $\beta = |\text{supp}(\vc^{(k)})|/m$ with randomly selected entries set to $10$. Finally, we implement QRK on this system with a quantile of size $q = 0.8$. In Figure~\ref{fig:big_beta} we plot the empirical behavior of QRK on systems with fraction of corruptions $\beta$ varying between $0.1, 0.15, 0.2,$ and $0.25$. We note that QRK converges to the threshold determined by the noise for all values of $\beta \le 1-q$.  When $\beta$ exceeds $1-q$, some corrupted indices are included in $B_k$ each iteration and cause QRK to diverge.

\begin{figure}
    \includegraphics[width=0.5\textwidth]{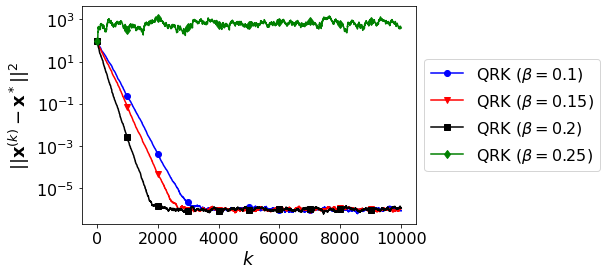}
    \caption{Empirical behavior of QRK with $q = 0.8$ on system defined by $A \in \mathbb{R}^{20000 \times 100}$ and $\vb^{(k)} = \vb + \vn^{(k)} + \vc^{(k)}$ where corruption $\vc^{(k)}$ is a vector with $\beta m$ uniformly randomly selected entries equal to the corruption size (10) and the remainder equal to 0. The noise $\vn^{(k)}$ is an i.i.d.\ random vector sampled in each iteration with entries from $\mathcal{N}(0,10^{-8})$. Plots exhibit error vs iteration on systems with fraction of corruption $\beta \in \{0.1, 0.15, 0.2, 0.25\}$.}
    \label{fig:big_beta}
\end{figure}

\subsection{QRK for corruption detection}

In this section, we compare the fraction of corrupted indices that are detected by the largest magnitude entries of the residual of QRK to the lower bound on the probability that all are detected provided by Corollary~\ref{cor:noisy corruption detection}.  We generate a system with $A \in \mathbb{R}^{20000 \times 100}$ generated with i.i.d.\ entries sampled from $\mathcal{N}(0,1)$, then row-normalize $A$, generate $\vx^* \in \mathbb{R}^{100}$ with i.i.d.\ entries sampled from $\mathcal{N}(0,100)$, and $\vb^{(k)} = A\vx^* + \vc^{(k)}$ where in each iteration we have corruption vector $\vc^{(k)}$ defined by $\beta = |\text{supp}(\vc^{(k)})|/m = 0.001$ with randomly selected entries set to $10$. Finally, we implement QRK on this system with a quantile of size $q = 0.6$.  We plot the fraction of corrupted indices that are included in the largest magnitude $\beta m$ entries of the residuals of the QRK iterates.  We also plot the lower bound on the probability that all $\beta m$ corrupted indices are included in the largest magnitude entries of the residuals of the QRK iterates, that is, the probability that these entries are easily detectable.  We plot these two values over all 8000 iterations of the QRK trial in Figure~\ref{fig:corr_detection}.  We note that the lower bound on the probability provided in Corollary~\ref{cor:noisy corruption detection} is a conservative estimate given how quickly the QRK residuals reveal the time-varying corrupted indices.  This bound is likely loose due to the looseness of the upper bound on the convergence rate.

\begin{figure}
    \includegraphics[width=0.8\textwidth]{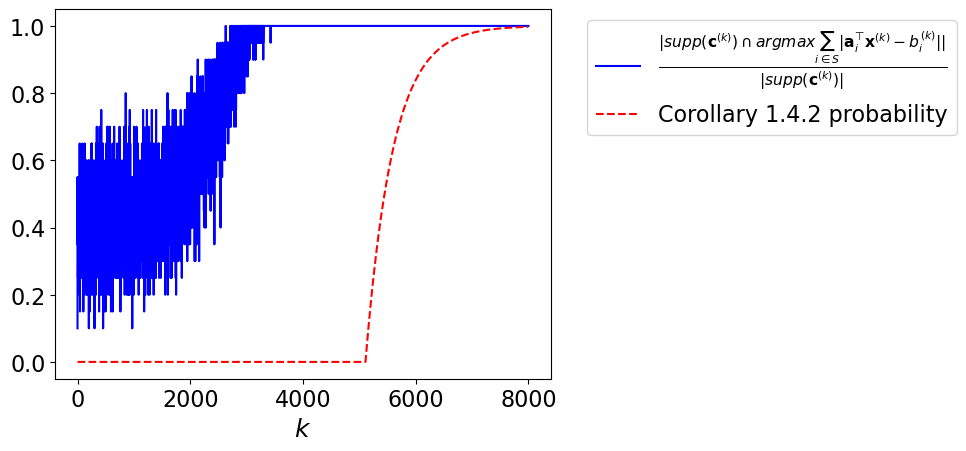}
    \caption{Fraction of corrupted indices in the largest magnitude entries of the QRK ($q = 0.6$) iterate residuals and lower bound on probability of detection provided by Corollary~\ref{cor:noisy corruption detection} on system defined by $A \in \mathbb{R}^{20000 \times 100}$ and $\vb^{(k)} = A\vx^* + \vc^{(k)}$ with corruption rate $\beta = 0.001$.}\label{fig:corr_detection}
\end{figure}

\subsection{RK on systems with time-varying noise}
In this section, we compare the upper bound provided by Theorem~\ref{thm:noisyRK} to the empirical behavior of RK on systems with noise vectors $\vn^{(k)}$ with entries drawn i.i.d.\ from a distribution with mean $\mu$ and standard deviation $s$.  We generate a system with $A \in \mathbb{R}^{20000 \times 100}$ generated with i.i.d.\ entries sampled from $\mathcal{N}(0,1)$, $\vx^* \in \mathbb{R}^{100}$ with i.i.d.\ entries sampled from $\mathcal{N}(0,1)$, and $\vb^{(k)} = A\vx^* + \vn^{(k)}$ where in each iteration $\vn^{(k)}$ is sampled from $\mathcal{N}(\mu,s^2)$ where $\mu \in \{0, 0.01, 0.1\}$ and $s \in \{0, 0.01, 0.1\}$ are noted.  In Figure~\ref{fig:thm2.1mu}, we plot the error of the RK method and the upper bounds on this error given in Theorem~\ref{thm:noisyRK} on these systems.  We note that these upper bounds are quite tight due to the consistent form of the sampled noise across iterations.

\begin{figure}
    \includegraphics[width=0.49\textwidth]{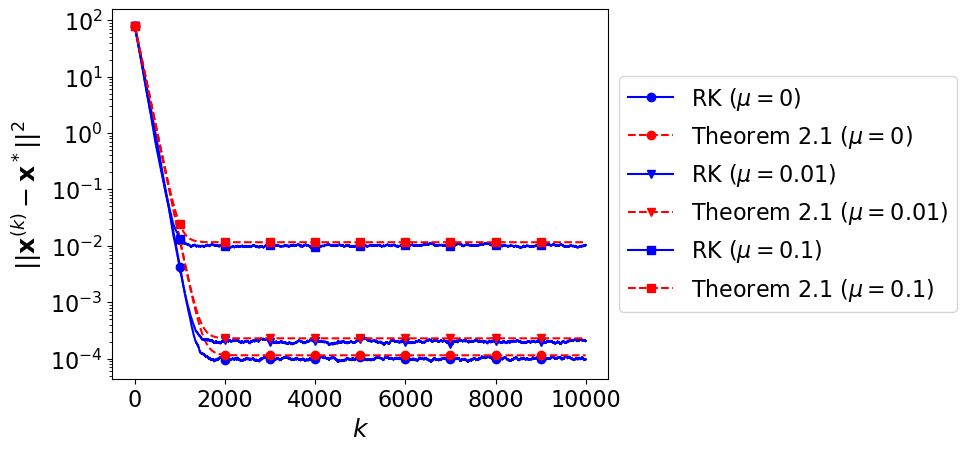}
    \includegraphics[width=0.49\textwidth]{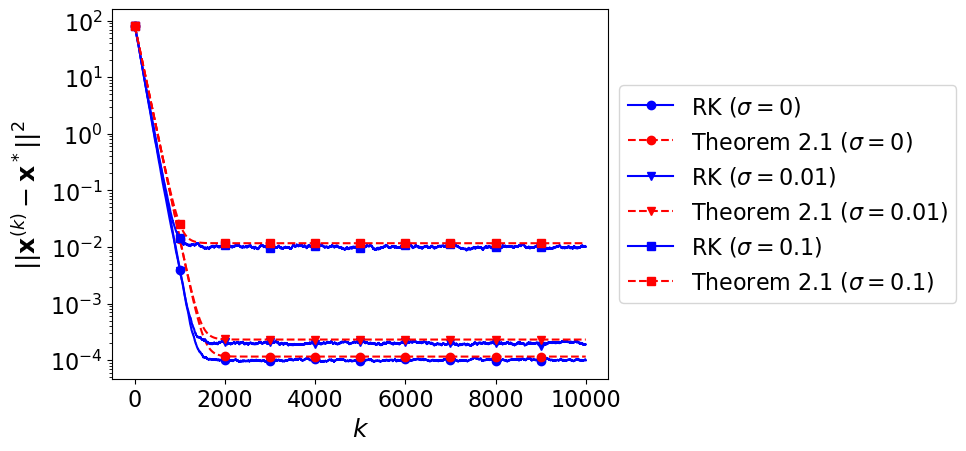}
    \caption{Empirical behavior of RK and upper bound provided in Theorem~\ref{thm:noisyRK} on system defined by $A \in \mathbb{R}^{20000 \times 100}$ and $\vb^{(k)} = A\vx^* + \vn^{(k)}$ where noise $\vn^{(k)}$ is an i.i.d.\ random vector sampled in each iteration with entries from $\mathcal{N}(\mu,s^2)$. The plot on the left shows behavior as $\mu$ varies between 0, 0.01, and 0.1, while $s=0.01$. The plot on the right shows behavior as $s$ varies between 0, 0.01, and 0.1, while $\mu=0.01$. Experiments are averaged over 10 trials.}\label{fig:thm2.1mu}
\end{figure}

\section{Conclusion}\label{sec:conclusion}

In this work, we prove that the QRK method converges even in the presence of \emph{time-varying} perturbation by noise and corruption.  We provide an upper bound on the error of the QRK iterates, which includes a convergence horizon term dictated by the magnitude of the time-varying noise and a rate term that depends upon the corruption rate of the time-varying corruption perturbing the system.
 We specialize this result to several noise models: bounded noise, random noise with given mean and standard deviation, and noise sampled from $\mathcal{N}(0,s^2)$.  We additionally use these results to provide a lower bound on the probability that the corrupted indices will be revealed by examining the QRK iterate residuals.

Our numerical experiments illustrate our theoretical results and reveal that many of our estimates are conservative.  These results reflect the qualitative behavior of the QRK method, but are quite loose, likely due to the looseness of the convergence rate estimate and the convergence horizon bound.

We highlight that our results are tailored to hold for time-varying and potentially \emph{adversarial} corruption.  A very interesting future question, given by Steinerberger~\cite{steinerberger2021quantile}, is whether QRK can be shown to converge for a larger fraction of corruptions, say near 50\%, when the corrupted indices are chosen at \emph{random}.  This corruption model seems very reasonable for a distributed computing application where corruption could occur during data access.

Another interesting future question is to develop a variant of RK that is robust to corruption and noise simultaneously, breaking the convergence horizon suffered by QRK in the presence of noise.  In the corruption-free case, the \emph{Randomized Extended Kaczmarz (REK)} method~\cite{zouzias2015randomized} utilizes a column-wise projection step, in addition to the usual RK projections, to break the convergence horizon and converge to the least-squares solution.  Providing such a method which is simultaneously robust to corruption would be a useful contribution.



\section*{Funding Statement}

The authors are grateful to and were partially supported by NSF DMS \#2211318.

\section*{Acknowledgements}
The authors are grateful to Anna Ma and Emeric Battaglia for useful conversations and helpful suggestions on this manuscript.



\bibliography{bibliography}
\bibliographystyle{abbrv}

\end{document}